\documentclass[a4paper,12pt]{article}
\usepackage[utf8]{inputenc}
\usepackage{amsmath}
\usepackage{amsfonts}
\usepackage{amssymb}
\usepackage{amsthm}
\usepackage[english]{babel}
\usepackage{fontenc}
\usepackage{graphicx}
\usepackage{mathtools}
\usepackage{accents}
\usepackage[normalem]{ulem}
\mathtoolsset{showonlyrefs}
\usepackage{dsfont}
\usepackage[pdftex,bookmarks=true,pdfborder={0 0 0}, breaklinks]{hyperref}
\usepackage[numbers,sort]{natbib}
\usepackage[hmargin=2.5cm,vmargin=2cm,paper=a4paper]{geometry}
\newcommand{\ls}{\leqslant}
\newcommand{\gs}{\geqslant}

\newcommand{\R}{\mathbb R}
\newcommand{\C}{\mathcal{C}}
\newcommand{\per}{\operatorname{Per}}

\newcommand{\supp}{\operatorname{supp}}
\DeclareMathOperator*{\esssup}{ess\,sup}
\DeclareMathOperator*{\essinf}{ess\,inf}
\newcommand{\eps}{\varepsilon}
\newcommand{\BV}{\mathrm{BV}}
\newcommand{\D}{{\widehat D}}
\newcommand{\Hm}{{\mathcal{H}}}
\newcommand{\f}{\widehat f}
\renewcommand{\u}{\widehat u}
\newcommand{\Om}{\widehat \Omega}
\newcommand{\Xih}{\widehat \Xi}
\newcommand{\otheta}{\overline \theta}
\newcommand{\dd}{\, \mathrm{d}}
\newcommand{\Id}{\operatorname{Id}}
\newcommand{\TV}{\operatorname{TV}}
\newcommand{\1}{\mathds 1}
\newcommand{\eqalpha}{\underline{\alpha}}

\newcommand{\scal}[2]{ \left \langle #1, \, #2 \right \rangle}
\newcommand\restr[2]{{
\left.\kern-\nulldelimiterspace #1 \vphantom{\big|} \right|_{#2} 
}}

\DeclareMathOperator*{\argmin}{arg\,min}

\newtheorem{theorem}{Theorem}
\newtheorem{prop}{Proposition}
\newtheorem{cor}{Corollary}

\newtheorem{lemma}{Lemma}
\theoremstyle{definition}

\newtheorem{remark}{Remark}
\newtheorem{example}{Example}
\usepackage{color}
\usepackage[dvipsnames]{xcolor}
\usepackage{comment}

\begin{document}
\title{Convergence of level sets in fractional Laplacian regularization\footnotetext{2020 Mathematics Subject Classification: 35R11, 47A52, 68U10, 35B51.}}
\author{Jos\'e A. Iglesias\thanks{Department of Applied Mathematics, University of Twente, Enschede, The Netherlands \newline(\texttt{jose.iglesias{@}utwente.nl})} , Gwenael Mercier\thanks{Computational Science Center, University of Vienna, Austria (\texttt{gwenael.mercier{@}univie.ac.at})}}
\date{}
\maketitle
\begin{abstract}
The use of the fractional Laplacian in image denoising and regularization of inverse problems has enjoyed a recent surge in popularity, since for discontinuous functions it can behave less aggressively than methods based on $H^1$ norms, while being linear and computable with fast spectral numerical methods. In this work, we examine denoising and linear inverse problems regularized with fractional Laplacian in the vanishing noise and regularization parameter regime. The clean data is assumed piecewise constant in the first case, and continuous and satisfying a source condition in the second. In these settings, we prove results of convergence of level set boundaries with respect to Hausdorff distance, and additionally convergence rates in the case of denoising and indicatrix clean data. The main technical tool for this purpose is a family of barriers constructed by Savin and Valdinoci for studying the fractional Allen-Cahn equation. To help put these fractional methods in context, comparisons with the total variation and classical Laplacian are provided throughout.
\end{abstract}

\section{Introduction}
In this work, and within the context of multidimensional data like natural images or recovered material parameters, we are interested in variational regularization approaches of the form
\begin{equation}\label{eq:introfunc}\min_u \|Au - f\|^2_{\Hm}+\alpha |u|^2_{H^s},\end{equation}
where $|u|^2_{H^s}$ is a fractional order Sobolev seminorm (see \eqref{eq:complementformula} below), $A$ is a linear operator and $\Hm$ is a Hilbert space containing the measurements. Owing to the fractional Laplacian appearing in its Euler-Lagrange equation, we refer to \eqref{eq:introfunc} as \emph{fractional Laplacian regularization}. Although it has a nonlocal character, the weights involved are singular and therefore are biased toward short range interactions. This is in contrast to patch-based methods designed to exploit long-range similarities across the image (see \cite{BuaColMor05, GilOsh08} for some prototypical examples).

Our focus is instead on the use of fractional order seminorms for a moderate amount of smoothing, and its compatibility with less regular inputs and solutions. Unlike other linear regularization methods like basic Tikhonov or $H^1$ seminorm regularization, one might claim that fractional Laplacian regularization with low order could be well adapted to images with distinct objects and discontinuities. This is the point of view adopted in works like \cite{AntBar17, AntDiKha20} and \cite{GatHes15}, mostly from a numerical perspective, in particular because using spectral methods could be extremely fast for such a problem. For example, in \cite{AntBar17} the inclusion \begin{equation}\label{eq:bvfractionalinclusion}\BV(\mathbb T^d) \cap L^\infty(\mathbb T^d) \subset H^s(\mathbb T^d)\end{equation} for $s<1/2$ (with $\mathbb T^d$ being the $d$-dimensional torus) is taken as a concrete argument in this direction, since the space $\BV$ of functions of bounded variation is the most important framework for problems with discontinuous solutions still allowing modelling using derivatives.

\begin{figure}[htb]
     \begin{center}
     \mbox{} 
     \hfill
	 \raisebox{-0.5\height}{\includegraphics[width=0.305\textwidth]{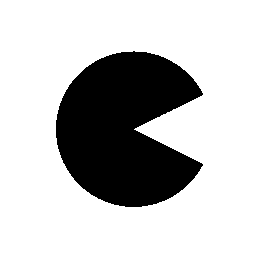}}
     \hspace{0.25cm}
	 \raisebox{-0.5\height}{\includegraphics[width=0.305\textwidth]{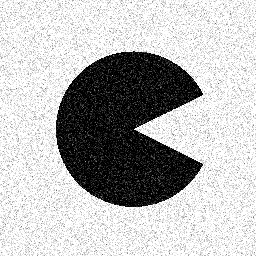}}
	 \hfill
	 \mbox{}
	 \\
	 \vspace{0.4cm}
	 \mbox{} 
     \hfill
	 \raisebox{-0.5\height}{\includegraphics[width=0.305\textwidth]{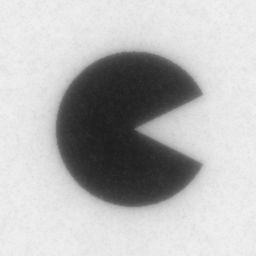}}
     \hspace{0.25cm}
	 \raisebox{-0.5\height}{\includegraphics[width=0.305\textwidth]{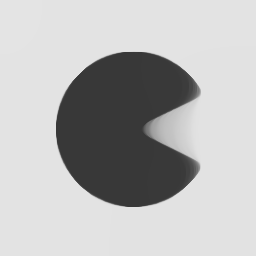}}
     \hspace{0.25cm}
	 \raisebox{-0.5\height}{\includegraphics[width=0.305\textwidth]{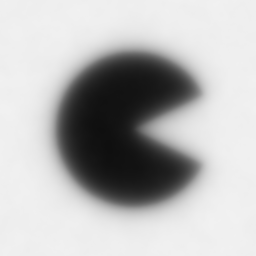}}
	 \hfill
	 \mbox{}
     \end{center}
     \caption{Oversmoothed denoising of a characteristic function. Upper row: clean and noisy images. Bottom row: Result with fractional Laplacian and $s=0.49$, with total variation, and with classical Laplacian. The $\TV$ result has been computed with the PDHG algorithm \cite{ChaPoc11} and finite differences, the others in the spectral domain with periodic boundary conditions by FFT. The large regularization parameters have been chosen so that $\text{PSNR}=16\pm0.01 \text{ dB}$ in all cases. In this case the fractional method gives a relatively sharp result, while being $\approx 10^3$ times faster than $\TV$.}\label{fig:pacman}
\end{figure}

In this work we explore this claim from an analytical point of view and beyond the spaces where the functionals may be defined, by considering geometric properties of the solutions arising from this type of regularization. In particular, we seek to answer questions like ``When denoising a characteristic function corrupted by noise, how strong is the smearing of edges caused by fractional Laplacian regularization?''. Figure \ref{fig:pacman} depicts an oversmoothed numerical example where the effect on edges of fractional Laplacian regularization can be seen easily, and also compared to total variation and $H^1$ regularization corresponding to the usual Laplacian.

To be able to give precise analytical answers, we turn to the low noise regime with vanishing regularization parameter, and formulate such results in terms of the Hausdorff distance between level set boundaries, which may be seen as uniform convergence of the objects in the images. This type of convergence is already known for $\TV$ regularization under various assumptions \cite{ChaDuvPeyPoo17, IglMerSch18, IglMer20, IglMer21} and, somewhat surprisingly, it is not only also true for fractional Laplacian regularization (see Theorem \ref{thm:hausmaxprinc} below), but holds as well for $H^1$ regularization. Where we do find a difference is in terms of convergence rates (Theorem \ref{thm:convrates}), since our proof depends on inclusions of the type \eqref{eq:bvfractionalinclusion}.

Beyond the case of piecewise constant ideal data and the recovery of jumps, we also consider the roughly opposite case of continuous ideal data. This allows us to study not only denoising but also regularization of linear inverse problems assuming the source condition and for a type of strongly smoothing operators, which in addition to being compact have their adjoint mapping continuously into $L^\infty$ (Theorem \ref{thm:hausmaxprincA}).

A limitation of our results is that we are forced to use noise belonging to $L^q$ with an exponent $q>2$  depending on the dimension and order of the regularization, or operators continuous on the dual space $L^p$ with $p=q'<2$. Roughly speaking, compared to a more standard $L^2$ situation, these assumptions mean that we force a stronger matching of observation and solution. Whether geometric results such as those proved here are possible without departing from a fully Hilbert space framework could be an interesting question for future work.

\subsection{Some notation} Within this article we work extensively with subsets of $\R^d$. In this setting, given $E, F \subset \R^d$ we use the notation $E^c := \R^d \setminus E$ for the complement, $\1_E:\R^d\to \R^+\cup \{0\}$ for the indicatrix taking the value $1$ on $E$ and $0$ on $E^c$, $|E|$ for the Lebesgue measure,
\[d(x,E) := \inf_{y \in E} |x-y|\]
for the distance from a point to a set, and 
\begin{equation}\label{eq:hausdorffdist}d_H(E,F):=\max\left(\sup_{x \in E} d(x,F), \ \sup_{y \in F} d(y,E)\right)\end{equation}
for the Hausdorff distance between two subsets. Moreover, we say that $E$ satisfies density estimates if there are some $C_E \in (0,1)$ and $r_0 \in (0,1)$ such that
\begin{equation}\label{eq:densityestA}\frac{|E \cap B(x,r)|}{|B(x,r)|}\gs C_E \text{ and }\frac{|B(x,r) \setminus E|}{|B(x,r)|}\gs C_E \text{ for all }r\ls r_0 \text{ and all } x \in \partial E.\end{equation}

\section{Fractional Laplacian regularization from a PDE perspective}
Throughout, we assume $\Omega\subset \R^d$ is a bounded Lipschitz domain satisfying the exterior ball condition, that is, there is some radius such that every point of $\partial \Omega$ can be touched with a ball of this radius contained in $\R^d\setminus \Omega$. Moreover, let $\mathcal H$ be a Hilbert space, $p \ls 2$ and $A: L^p(\Omega) \to \mathcal H$ a bounded linear operator. We want to invert $Au=f$ and minimize, for $n$ a given noise instance and $\alpha>0$ a regularization parameter, the functional
\begin{equation}\Vert Au - f -n \Vert_{\Hm}^2 + \alpha |u|_{H^s}^2 \label{eq:mingen} \end{equation}
among
\[u \in H^s_0(\overline{\Omega}) := \{ u: \R^d \to \R \, \mid \, \restr{u}{\R^d \setminus \Omega} = 0, \ |u|_{H^s} < + \infty \},\]
where the Gagliardo-Slobodeckij seminorm of (fractional) order $s\in (0,1)$ in $\R^d$ is defined as
\[|u|^2_{H^s} := |u|^2_{H^s(\R^d)} := \int_{\R^d} \int_{\R^d} \frac{|u(x)-u(y)|^2}{|x-y|^{d+2s}} \dd x \dd y,\]
where as noted we skip the domain $\R^d$ in the notation $|\cdot|_{H^s}$, since in the sequel, every fractional seminorm we use will be computed in the full space. Moreover, we also note that $H^s_0(\overline{\Omega})$ differs from $H^s_0(\Omega)$ defined as the closure of $\C^\infty_c(\Omega)$ in the $H^s(\Omega)$ topology, which we do not use in this article.

Now, the space $H^s_0(\overline{\Omega})$ is a Hilbert space (see \cite[Lem.~7]{SerVal12}, for example) with the inner product $\scal{u_1}{u_2}_{L^2} + \scal{u_1}{u_2}_{H^s}$ defined by
\[ \scal{u_1}{u_2}_{H^s} :=  \int_{\R^d} \int_{\R^d} \frac{ (u_1(x)-u_1(y) ) (u_2(x)-u_2(y) )}{|x-y|^{d+2s}} \dd x \dd y. \]
Since $u$ is constrained to vanish on the complement of $\Omega$, we have
\begin{equation}\label{eq:complementformula}\begin{aligned}|u|^2_{H^s} &= \iint_{(\R^d \times \R^d) \setminus (\Omega^c \times \Omega^c)} \frac{|u(x) - u(y)|^2}{|x-y|^{d+2s}} \dd x\dd y\\
&= \int_{\Omega}\int_{\Omega} \frac{|u(x) - u(y)|^2}{|x-y|^{d+2s}} \dd x\dd y \,+\, 2 \int_{\Omega} \int_{\Omega^c} \frac{|u(x)|^2}{|x-y|^{d+2s}} \dd x\dd y.\end{aligned}\end{equation}

We also recall (see \cite[Thm.~2.2.1]{BucVal16}, for example) the Sobolev inequality applied to $u \in H^s_0(\overline{\Omega})$,
\begin{equation}
\label{eq:sobineq}
    \Vert u \Vert_{L^{2d/(d-2s)}(\R^d)} \ls \Theta |u|_{H^s}.
\end{equation}

Let us remark that in the above definitions there is a peculiarity that is common when working with nonlocal equations: we are interested in functions supported in $\Omega$, but the $H^s$ norm we consider involves integrals over the whole $\R^d$, since the interaction kernel $1/|x-y|^{d+2s}$ is not compactly supported. For this reason, we make the convention that
\[L^p(\Omega) = \left\{ u \in L^p(\R^d) \,\middle\vert\, u=0 \text{ on }\R^d\setminus \Omega\right\},\]
noting that the identification by extension and restriction is clearly well-defined.

Let us now derive the Euler-Lagrange equation for this functional. We want to compare the energy of a minimizer $u_{\alpha,n}$ of \eqref{eq:mingen} with the energy of an admissible perturbation $u_{\alpha,n}+t h$ with $h\in H_0^s(\overline{\Omega})$ and $t \in \R$. This reads (we skip the indices for simplicity)
\begin{equation}\label{eq:perturbation}\Vert Au - f -n \Vert_{\Hm}^2 + \alpha \Vert u \Vert_{H^s}^2  \ls \Vert A(u+th) - f -n \Vert_{\Hm}^2 + \alpha \Vert u+th \Vert_{H^s}^2.\end{equation}
The right hand side writes
\[\Vert Au - f -n \Vert_{\Hm}^2 +t^2 \Vert Ah \Vert_{\Hm}^2 + 2t \scal{Au-f-n}{Ah}_{\Hm}+  \alpha \left( \Vert u \Vert_{H^s}^2 + t^2\Vert h \Vert_{H^s}^2 + 2t \scal{u}{h}_{H^s}  \right) \]
which implies, since \eqref{eq:perturbation} needs to be true for all $t$,
\[ \scal{Au-f-n}{Ah}_{\Hm} =-\alpha \scal{u}{h}_{H^s} \]
which means, writing explicitly the inner products and using the adjoint $A^\ast: \mathcal H \to (L^p(\Omega))'=L^q(\Omega)$, that for any $h \in L^p(\Omega)$ (we can write the first integral on the whole $\R^d$, since $h=0$ outside of $\Omega$)
\begin{equation}
    \int_{\R^d} A^\ast(f+n-Au) h = \alpha \iint_{\R^d \times \R^d} \frac{ (u(x)-u(y) ) (h(x)-h(y) )}{|x-y|^{d+2s}} \dd x \dd y.
    \label{eq:ELweak}
\end{equation}
Now, we notice that $H^1_0(\overline{\Omega}) \subset L^2(\Omega) \subset L^p(\Omega)$, and \eqref{eq:ELweak} holding for all $h \in H^1_0(\overline{\Omega})$ is defined to be (see \cite{LeoPerPriSor15, Ros16}, for example) the weak formulation of
%Since this is valid for any $h$, if we had that $u \in H^{2s}(\R^d)$ the equation could be written as
\begin{equation}\label{eq:ELgen}\begin{cases}\eqalpha(-\Delta)^s u = A^\ast \left(f+n-Au\right) & \text{ on }\Omega \\ u = 0 & \text{ on } \Omega^c = \R^d \setminus \Omega.\end{cases}\end{equation}
Here, $(-\Delta)^s$ denotes the integral fractional Laplacian on $\R^d$, defined for $u$ regular enough (say in the Schwartz space of rapidly decaying $\mathcal{C}^\infty$ functions) as
\begin{equation}\label{eq:fraclap}\begin{gathered}(-\Delta)^s u(x) = C(d,s) \lim_{\delta \to 0}\int_{\R^d \setminus B(x, \delta)} \frac{u(x) - u(y)}{|x-y|^{d+2s}} \dd y, \text{ for }\\C(d,s):=\left( \int_{\R^d}\frac{1-\cos(y \cdot e_1)}{|y|^{d+2s}} \dd y \right)^{-1},\end{gathered}\end{equation}
for $e_1$ the first unit vector of the canonical basis of $\R^d$, and where we use the notation $\eqalpha := 2 C(d, s)^{-1} \alpha$. It might seem inconvenient that the factor $C(d,s)$ appears between the regularization energy \eqref{eq:mingen} and the fractional PDE \eqref{eq:ELgen}, but this constant (for which we use the conventions of \cite{DinPalVal12}) is necessary for a few reasons. First, it is required to maintain the relation of the operator $(-\Delta)^s$ with the classical Laplacian $-\Delta$, see \cite[Prop.~4.4]{DinPalVal12} for the pointwise limit $(-\Delta)^s u \to (-\Delta) u$ when $s \to 1$ and smooth $u$. Moreover, it is also required to maintain consistency with the definition in terms of Fourier multipliers, in which we have \cite[Prop.~3.3]{DinPalVal12} the relation
\[(-\Delta)^s u = \mathcal{F}^{-1}\left(|\xi|^{2s} \mathcal{F} u\right),\]
which also tells us \cite[Prop.~3.6]{DinPalVal12} that for any given $u \in L^2(\R^d)$, $(-\Delta)^s u \in L^2(\R^d)$ if and only if $u \in H^{2s}(\R^d)$. Moreover, we also have the identity
\begin{equation}\label{eq:parseval}\scal{u_1}{u_2}_{H^s} = 2\, C(d,s)^{-1} \scal{(-\Delta)^{s/2} u_1}{(-\Delta)^{s/2} u_2}_{L^2(\R^d)},\end{equation}
which allows us to interpret appearances of $\scal{\cdot}{\cdot}_{H^s}$ in convergence estimates in later sections. In particular, the two formulas above and the Plancherel theorem  yield for $u_1 \in H^s(\R^d)$ and $u_2 \in H^{2s}(\R^d)$ that
\begin{equation}\label{eq:intbyparts}\begin{aligned}\scal{u_1}{u_2}_{H^s} &= 2 C(d,s)^{-1}\! \big\langle (-\Delta)^{s/2} u_1, (-\Delta)^{s/2} u_2 \big\rangle_{L^2(\R^d)} \\ &= 2 C(d,s)^{-1}\! \scal{u_1}{(-\Delta)^{s} u_2}_{L^2(\R^d)},\end{aligned}\end{equation}
which further justifies speaking of \eqref{eq:ELweak} as a weak formulation.

In what follows, when speaking about solutions to \eqref{eq:ELgen}, we will always mean functions which satisfy the weak formulation \eqref{eq:ELweak}. Moreover, given functions $u,f$ defined on $\R^d$, we write for brevity
\begin{equation}(-\Delta)^s u + u = f \text{ in } \mathcal O \text{ for an open set }\mathcal{O}\end{equation}
whenever we have 
\begin{equation}\label{eq:laplacianlocal}\iint_{\R^d \times \R^d} \frac{ (u(x)-u(y) ) (h(x)-h(y))}{|x-y|^{d+2s}} \dd x \dd y = \int_{\R^d} (f-u) h \text{ for all }h \in H^s_0(\overline{\mathcal{O}}).\end{equation}
Notice that if $\mathcal{O} \subset \Omega$, by definition we have $H^s_0(\overline{\mathcal{O}}) \subset H^s_0(\overline{\Omega})$. This means that when $A=\Id$, a weak solution of \eqref{eq:ELgen} satisfies in particular that $(-\Delta)^s u = (f-u)/\alpha \text{ in } \mathcal O $ for all $\mathcal{O}\subset \Omega$.
We write similarly that 
\begin{equation}(-\Delta)^s u + u \ls f \text{ in } \mathcal O\end{equation}
whenever it holds that 
\begin{equation}\label{eq:sublaplacianlocal}\iint_{\R^d \times \R^d} \frac{ (u(x)-u(y) ) (h(x)-h(y))}{|x-y|^{d+2s}} \dd x \dd y \ls \int_{\R^d} (f-u) h \text{ if }h \in H^s_0(\overline{\mathcal{O}}) \text{ with } \, h\gs 0.\end{equation}

% We will also speak about sub (resp. super)-solutions, which will satisfy the inequality
% \begin{equation}\alpha \iint_{\R^d \times \R^d} \frac{ (u(x)-u(y) ) (h(x)-h(y) )}{|x-y|^{d+2s}} \dd x \dd y \ls \int_{\R^d} A^\ast(f+n-Au) h \label{eq:ELsub}  \end{equation}
% (resp $\gs$) for every $h \gs 0$.
Notice that the expressions in both \eqref{eq:fraclap} and \eqref{eq:complementformula} contain contributions from outside the support of $u$. The combination of this unbounded domain of interaction and the singularity of the kernel gives rise to numerical challenges, particularly in comparison with the periodic case in which straightforward spectral methods are applicable. In any case, recent numerical works tackle the efficient computation of problems with the integral fractional Laplacian we consider here, see \cite{AntDonStr21} where image denoising by a Dirichlet problem of the form \eqref{eq:ELgen} with $A = \Id$ and its comparison to the periodic case are considered.

\subsection{Comparison principle}
Many of our results concern the particular case of denoising, in which the data and noise are assumed to belong to $L^2(\Omega)$ and $\Hm = L^2(\Omega)$ which allows for $A$ to be simply the identity. In this setting, let us now recall the following statement for the fractional Laplacian. Our proof follows that of \cite[Prop.~4.1]{Ros16}.
\begin{prop}[Weak maximum principle]
Let $\mathcal O$ be an open subset and $u \in H^s(\R^d)$ satisfy
\begin{equation}(-\Delta)^s u + u \ls 0 \text{ in } \mathcal O \label{eq:maxpple}\end{equation}
as well as $u \ls 0$ on $\R^d \setminus \mathcal O$.
Then, we also have $u \ls 0$ (a.e.) on $\mathcal O$.
\label{prop:maxpple}
\end{prop}
\begin{proof} Because $u \in H^s(\R^d)$ and $u \ls 0$ on $\R^d \setminus \mathcal O$, we have that $u^+ := \max(u,0) \in H^s_0(\overline{\mathcal{O}})$, so we may use it in \eqref{eq:sublaplacianlocal} with $f=0$ to find
\[ \iint_{\R^d \times \R^d} \frac{ (u(x)-u(y) ) (u^+(x)-u^+(y) )}{|x-y|^{d+2s}} \dd x \dd y \ls- \int_{\R^d} u u^+.\]
Splitting $u = u^+ - u^-$, we obtain
\[\iint_{\R^d \times \R^d} \frac{ (u^+(x)-u^+(y) )^2}{|x-y|^{d+2s}}-\frac{ (u^-(x)-u^-(y) ) (u^+(x)-u^+(y) )}{|x-y|^{d+2s}} \dd x \dd y \ls- \int_{\R^d} (u^+)^2. \]
Now, we can notice that $ (u^-(x)-u^-(y) ) (u^+(x)-u^+(y) ) \ls 0 $ which implies that the left hand side of the last inequality is nonnegative. This forces $u^+=0$ a.e., that is $u\ls 0$ in the whole $\R^d$.
\end{proof}

\begin{cor}[Comparison principle]
Let $f,g$ be two $L^2$ functions with $f \ls g$ a.e.~on an open subset $\mathcal O\subset \R^d$ and $u$ and $v$ solutions of the respective equations
\[ (-\Delta)^s u + u = f \quad \text{and}\quad  (-\Delta)^s v + v = g \quad \text{ on } \mathcal O.\] 
If $u\ls v$ on $\R^d \setminus \mathcal O$, then this inequality also holds a.e.~on $\mathcal O$.
\label{cor:comppple}
\end{cor}
\subsection{Regularity for solutions of fractional Dirichlet problems}\label{sec:regularity}
A particularity of our formulation \eqref{eq:mingen} above, is that we have stepped out of a purely Hilbert space formulation where $A:L^2(\Omega) \to L^2(\Omega)$, which is both natural and commonly used, in particular in \cite{AntDiKha20} for regularization with fractional Laplacian. From now on, whenever $A$ appears (that is, when the problem considered is not a simple denoising), we will assume that it is defined instead on $L^p(\Omega)$ with \[p<\frac{d-2s}{d}<2.\] 
The reason for this are the following boundedness and regularity results, which hold only with right hand side with large enough integrability (in the dual space $L^q(\Omega)$). Needing $L^\infty$ control on the solutions is quite natural for the kind of results we want to prove, since we would like to work with level sets of the minimizers, and for that one should know at which values these level sets are to be examined. Let us remark that the situation in previous works proving convergence results for level sets in $\TV$ regularization is similar: it is proved in \cite{BreIglMer2X} that the assumptions used in \cite{ChaDuvPeyPoo17, IglMerSch18, IglMer21} all lead to $L^\infty$ estimates independent of $\alpha$.

We consider now the regularity of solutions of Dirichlet problems of the type
\begin{equation}\label{eq:fraclin}\begin{cases}(-\Delta)^s u + u = g &\text{ on }\Omega\\u= 0 &\text{ on }\R^d \setminus\Omega.\end{cases}\end{equation}
In what follows, we will repeatedly use the following boundedness result:
\begin{prop}\label{prop:boundedsol}
Let $g \in L^q$ with $q > d/2s$ and $u$ the weak solution to \eqref{eq:fraclin}. There exists a constant $C_S(\Omega,q,d,s)$ so that
\[ \Vert u \Vert_{L^\infty} \ls C_S(\Omega, q, d, s) \Vert g \Vert^{2d/(d-2s)}_{L^q(\Omega)}.\]
\end{prop}
\begin{proof}It follows by the classical Stampacchia method. A fractional version without linear term appears in \cite[Theorem 13]{LeoPerPriSor15}. However, their statement might make the reader think that the constant depends on $u$, whereas in the proof one sees that this dependence is just in terms of $|\supp u| \ls |\Omega|$. To show this dependence and that the estimate also holds with the linear term in the equation, we briefly present the complete proof.

As in \cite{LeoPerPriSor15} and in the classical case (see \cite[Thm.~B.2]{KinSta00}, for example) we introduce the soft thresholding function
\[ G_k(\sigma) = (\sigma-k)^+ - (\sigma+k)^-\]
and use, for a weak solution $u$ of \eqref{eq:fraclin}, the test function $h = G_k(u)$ in the variational formulation, which results in 
\begin{equation} \iint_{\R^d \times \R^d} \frac{\big(u(x)-u(y)\big)\big(G_k(u(x))-G_k(u(y))\big)}{|x-y|^{d+2s}} \dd x \dd y = \int_{\R^d} (g-u)G_k(u). \label{eq:varstamp}\end{equation}
Now, let us notice that we always have
\[u(x) (u(x)-k)^+ = (u(x)-k) (u(x)-k)^+ + k (u(x)-k)^+ = [(u(x)-k)^+]^2 +k (u(x)-k)^+ \]
as well as 
\[u(x) (u(x)+k)^- = (u(x)+k) (u(x)+k)^- - k (u(x)+k)^- = -[(u(x)+k)^-]^2 -k (u(x)+k)^- \]
which we can combine as
\[ u(x) G_k(u(x)) = G_k(u(x))^2 +k |G_k(u(x))|. \]
Finally, noticing that $|u| \gs |G_k(u)|$ and that $G_k$ preserves signs, we get that if $u(x)$ and $u(y)$ have different signs, then
\[u(x) G_k(u(y)) + u(y) G_k(u(x)) \ls 2 G_k(u(x)) G_k(u(y)),\]
while if $u(x)$ and $u(y)$ are both positive
\begin{align*}u(x) G_k(u(y)) + u(y) G_k(u(x)) &\ls \big(G_k(u(x))+k\big) G_k(u(y)) + \big(G_k(u(y))+k\big) G_k(u(x)) \\&\ls 2G_k(u(x))G_k(u(y)) + k \left( G_k(u(x)) + G_k(u(y)) \right).\end{align*}
and if they are both negative (in which case $u \gs G_k(u)-k$) then
\begin{align*}u(x) G_k(u(y)) + u(y) G_k(u(x)) &\ls\big(G_k(u(x))-k\big) G_k(u(y)) + \big(G_k(u(y))-k\big) G_k(u(x)) \\&\ls 2G_k(u(x))G_k(u(y)) + k \left( |G_k(u(x))| + |G_k(u(y)|) \right).\end{align*}
These three cases, combined with the previous equality, imply that
\[\iint_{\R^d \times \R^d} \frac{ (u(x)-u(y) )\big(G_k(u(x))-G_k(u(y))\big)}{|x-y|^{d+2s}} \dd x \dd y \gs \left|G_k(u)\right|^2_{H^s}.\]
Now, we can plug this estimate in \eqref{eq:varstamp} and note that $-uG_k(u) \ls 0$ to conclude
\[ \left|G_k(u)\right|^2_{H^s} \ls \int_{\R^d} gG_k(u) = \int_{\{|u| \gs k\}}gG_k(u).\]
The Sobolev inequality \eqref{eq:sobineq} on the left hand side and a Hölder inequality on the right implies, since $1-1/q-(d-2s)/2d=(qd-2d+2qs)/2qd$, 
\[\Vert G_k(u) \Vert^2_{L^{2d/(d-2s)}} \ls \Theta^2 \Vert g \Vert_{L^q} \Vert G_k(u) \Vert_{L^{2d/(d-2s)}} |\{|u| \gs k\}|^{\frac{qd-2d+2qs}{2qd}}.\]
Now if $k'>k$, we have $\{|u| \gs k'\} \subset \{|u| \gs k\}$ and $G_k(u) \1_{\{|u| \gs k'\}} \gs (k'-k)\1_{\{|u| \gs k'\}}$, so 
\[\Vert G_k(u) \Vert_{L^{2d/(d-2s)}} \gs (k'-k)|\{|u| \gs k'\}|^{(d-2s)/(2d)}.\]
This leads to, as soon as $k'>k$,
\[(k'-k)|\{|u| \gs k'\}|^{(d-2s)/(2d)}\ls \Theta^2 \Vert g \Vert_{L^q} |\{|u| \gs k\}|^{\frac{qd-2d+2qs}{2qd}},\]
which we can also write as
\[(k'-k)^{2d/(d-2s)}|\{|u| \gs k'\}|\ls \Theta^{\frac{4d}{d-2s}} \Vert g \Vert_{L^q}^{2d/(d-2s)} |\{|u| \gs k\}|^{\frac{qd-2d+2qs}{2qd}\frac{2d}{d-2s}}.\]
Now, in fact 
\[\delta:=\frac{qd-2d+2qs}{2qd}\frac{2d}{d-2s} > 1 \text{ if and only if }q>\frac{d}{2s},\]
which was the assumption on $q$. The previous inequality can then be written as
\[|\{|u| \gs k'\}|\ls \Theta^{\frac{4d}{d-2s}} \Vert g \Vert_{L^q}^{2d/(d-2s)} \frac{|\{|u| \gs k\}|^\delta}{(k'-k)^{2d/(d-2s)}},\]
and the standard extinction lemma \cite[Lem.~B.1]{KinSta00} implies that for 
\[k_0:=2^{\frac{\delta}{\delta-1}\frac{2d}{d-2s}} \Theta^{\frac{4d}{d-2s}} \Vert g \Vert_{L^q}^{2d/(d-2s)} |\{|u| \gs 0\}|^{\delta - 1} \] 
we have $|\{|u| \gs k\}|=0$ for $k \gs k_0$, and $k_0$ is the uniform bound on $u$ we were after. Note that since $u$ is supported on $\Omega$ we have $|\{|u| \gs 0\}|\ls |\Omega|$, so this bound depends on $g,\Omega, q,d$ and $s$ but not on the solution $u$ itself.
\end{proof}

From Proposition \ref{prop:boundedsol}, if $g \in L^q$ with $q>d/2s$ we see in particular that $u$ is a weak solution to 
\[\begin{cases}(-\Delta)^s u = g-u \in L^q(\Omega) &\text{ on }\Omega\\ u = 0 &\text{ on }\R^d\setminus \Omega\end{cases},\]
which (see \cite[Thm.~1.1]{Now21} for a very general but directly applicable statement) implies interior H\"older regularity for $u$, that is 
\begin{equation}\label{eq:uishoelder}u \in \C^{0,\gamma}\big(B(x,r)\big)\text{ for }\gamma \in (0, 2s-d/q)\end{equation}
and all balls $B(x, r) \subset \Omega$. With the additional assumption $g \in L^\infty$, regularity up to the boundary also holds but with the H\"older exponent saturating at $s$ due to boundary effects (see \cite{RosSer14} and \cite{Ros16} for further discussion):
\begin{prop}[{\cite[Prop.~1.1]{RosSer14}}]\label{prop:boundaryregularity}
Let $\Omega$ be a bounded Lipschitz domain satisfying the exterior ball condition, $g \in L^\infty$ and $u$ the weak solution to \eqref{eq:fraclin}. Then for some $C_G>0$ we have the estimate
\[\Vert u \Vert_{\C^{0,s}(\R^d)} \ls C_G \Vert g \Vert_{L^\infty(\Omega)}.\]
\end{prop}
Notice that in particular this result implies that $u(x)=0$ for all $x \in \partial \Omega$, a fact that we will use in Sections \ref{sec:indicatrix} and \ref{sec:flat} below. For $s$ small this is not a given, since functions in $H^s_0(\overline{\Omega}) \subset H^s(\R^d)$ may jump across $\partial \Omega$, see Section \ref{sec:fractv}.

\section{Relations to convex regularization theory}

The regularization functional \eqref{eq:mingen} that we consider is clearly quadratic, leading to linear optimality conditions. However, in this article it is more convenient for us to think of it as a general linear inverse problem with convex regularization, in the spirit of \cite{SchGraGroHalLen09, BurOsh04, SchKalHofKaz12}. This point of view and its explicit use of optimality conditions provides us with enough information on the behaviour of minimizers of \eqref{eq:mingen} with respect to $n$ and $\alpha$ in order to prove our PDE/geometric statements.

We now present the building blocks from regularization theory (specialized to the fractional Laplacian context) that we will need for our main results about level sets in Sections \ref{sec:indicatrix} and \ref{sec:convtocont}. Specifically, in Section \ref{sec:optimality} we make precise the fractional PDE meaning of the optimality conditions while Section \ref{sec:fractv} contains additional material on their relation with fractional perimeters. Afterwards, in Sections \ref{sec:convsub} we treat convergence of subgradients and convergence rates in Bregman distance. Finally, Section \ref{sec:L2rates} contains a basic result for convergence rates of denoising which we later use in Section \ref{sec:rates}.

\begin{prop}\label{prop:existence}
There is a unique minimizer $u_{\alpha,n}$ of \eqref{eq:mingen} in $L^p(\Omega) \cap H^s_0(\overline{\Omega})$, which is determined by the optimality condition \eqref{eq:ELweak}.
\end{prop}
\begin{proof}
We can just use the direct method: notice that
\[p<\frac{d}{d-2s}<2<\frac{2d}{d-2s},\]
we have the Sobolev inequality \eqref{eq:sobineq}, and the domain $\Omega$ is bounded. Since this functional is strictly convex, there can only be one minimizer.
\end{proof}

\subsection{Optimality and source conditions for fractional Laplacian regularization}\label{sec:optimality}
Let us start with the optimality condition for the functional \eqref{eq:mingen} at its unique minimizer $u_ {\alpha, n}$, which follows directly by the definition of the adjoint $A^\ast$ and subdifferential \cite[Def.~I.5.1]{EkeTem99} and reads
\[-\frac{1}{\alpha} A^\ast (Au_{\alpha, n}-f-n) \in \frac12 \partial_{L^p(\Omega)}|\cdot|^2_{H^s}(u_{\alpha, n}).\]
In the above, and denoting $u:=u_{\alpha, n}$ for simplicity, the subgradient is understood as the subset of $(L^p(\Omega))' = L^q(\Omega)$ defined as
\begin{equation}\label{eq:defsubgrad}v \in \partial_{L^p} |\cdot|^2_{H^s}(u) \text{ if and only if }|u+h|^2_{H^s} \gs |u|^2_{H^s} + \int_\Omega v h \text{ for all }h\in L^p(\Omega).
\end{equation}
We first remark that since $p \ls 2$ and the domain $\Omega$ is bounded, we have the embedding $L^2 \subset L^p$ and therefore $\partial_{L^p} |\cdot|^2_{H^s} \subset L^q \subset L^2$. Moreover, in the definition \eqref{eq:defsubgrad} all $h \in L^2$ are allowed as well, so $\partial_{L^p} |\cdot|^2_{H^s} \subset \partial_{L^2} |\cdot|^2_{H^s}$. Now, given $u \in L^p(\Omega)$ with $|u|_{H^s} < +\infty$ (which in fact implies $u \in L^2$) let us simply consider $h \in L^2(\Omega)$. Then, either $\vert h \vert_{H^s} = +\infty$ and there is nothing to check, or $\vert h \vert_{H^s} < +\infty$ and \eqref{eq:defsubgrad} implies
\begin{equation}\begin{aligned}\frac{1}{t}\big(\vert u&+th \vert^2_{H^s} - \vert u \vert_{H^s}^2\big) - \int_\Omega v h \\&= 2 \iint_{(\R^d \times \R^d) \setminus (\Omega^c \times \Omega^c)} \frac{(u(x) - u(y))(h(x)-h(y))}{|x-y|^{d+2s}} \dd x\dd y + t\vert h \vert^2_{H^s} - \int_\Omega v h \gs 0, \end{aligned}\end{equation}
which by taking the limit as $t \to 0$ and considering also $-h$, leads to 
\begin{equation}\label{eq:weakforsubg}2 \iint_{(\R^d \times \R^d) \setminus (\Omega^c \times \Omega^c)} \frac{(u(x) - u(y))(h(x)-h(y))}{|x-y|^{d+2s}} \dd x\dd y = \int_\Omega v h \text{ for all }h \in H^s_0(\overline{\Omega})\end{equation}
which is defined to be the weak formulation of
\[\begin{cases}2C(d,s)^{-1}(-\Delta)^s u = v/2 &\text{ on }\Omega\\u = 0 &\text{ on }\R^d \setminus \Omega.\end{cases}\]
Therefore we can consider our regularization problem in the framework of the previous section, that is as a standard weak formulation of a fractional PDE with right hand side in different Lebesgue spaces. Moreover, when making this connection we see that the subgradients of the regularization term at the minimizers are the central quantities of interest, which motivates investigating their convergence as $\alpha \to 0$, which we do in the next subsection.

This type of formulation also applies directly to the standard \emph{source condition} in the context of the functional \eqref{eq:mingen}. Such a condition is satisfied at the point $u^\dag \in L^p(\Omega)$ if there exists $z \in \Hm$ with $A^\ast z \in \partial_{L^p} \left[\frac12 |\cdot|^2_{H^s}(u^\dag)\right]$, so it is satisfied precisely when the weak formulation of 
\[\begin{cases}2 C(d,s)^{-1}(-\Delta)^s u^\dag = A^\ast z &\text{ on }\Omega\\u^\dag = 0 &\text{ on }\R^d \setminus \Omega.\end{cases}\]
holds. Now, in case we have such a source condition and $p<d/(d-2s)$ then we have that $A^\ast z \in L^q(\Omega)$ with $q=p'>d/2s$ and the results of Section \ref{sec:regularity} are applicable, so that $u^\dag$ is not just bounded but also H\"older continuous.

\subsection{Relations with fractional perimeters}\label{sec:fractv}
In previous works on geometric convergence for total variation regularization \cite{ChaDuvPeyPoo17, IglMerSch18, IglMer20, IglMer21} the dual variables or subgradients also play a central role. In that case, using the coarea formula the subgradients appear as perturbations in perimeter minimization problems for the level sets of $u_{\alpha, n}$, and control their regularity. 

On the other hand for the characteristic function $\1_D$ of a set $D$ and $s < 1/2$ we have the relation (following the notation of \cite[Sec.~1]{FigEtAl15}, for example) 
\[\frac12 \big|\1_D\big|^2_{H^s} = \frac12 \int_{\R^d}\int_{\R^d} \frac{|\1_D(x) - \1_D(y)|}{|x-y|^{d+2s}}\dd x \dd y = \int_{D}\int_{D^c} \frac{1}{|x-y|^{d+2s}}\dd x \dd y =: \per_{2s}(D),\]
and sets which are minimizers and almost-minimizers of the fractional perimeter $\per_{2s}$ also satisfy regularity properties. Specifically, for minimizers it is known that $\partial D \in \C^{1,\beta}$ for all $\beta<s$ outside a singular set of dimension at most $d-3$, see \cite[Thm.~6.1]{CafRoqSav10}, \cite[Cor.~2]{SavVal13} and the introduction to \cite{BarFigVal14}. Moreover, in \cite[Cor.~3.5]{FigEtAl15} H\"older regularity of the normal vector is also proved for flat almost-minimizers (in the sense of perturbations with a mass term). There is a limit to which kind of perturbations are allowed, though. Note that if we had 
\[D \in \argmin_E \,\per_{2s}(E) - C\int_E f\]
and $f \in L^q(\Omega)$ for some $q>d/2s$, then the fractional isoperimetric \cite[(1.1)]{FigEtAl15} and H\"older inequalities make the first term dominate the second. However, if instead $f \notin L^{d/2s}_{\text{loc}}(\Omega)$, the functional could assign low energy values to sets of vanishing mass.

Now, we might ask if there is any relation between regularity of almost-minimizers of fractional perimeter and the fractional Laplacian regularization we consider. Working in $H^s(\R^d)$ the coarea formula is not available, but we can reinterpret the subgradient as
\[u^\dag \in \argmin_u \frac12 |u|^2_{H^s} - \int_{\Omega} \big(A^\ast z \big)u,\]
which if $u^\dag = \1_D$, trying the above minimality with functions of the form $u=\1_E$ we end up with
\[D \in \argmin_E \,\per_{2s}(E) - C\int_E A^\ast z.\]
We note however that on the one hand that if $A^\ast z \in L^2(\Omega)$, then we have \cite[Thm.~1.4]{BicWarZua17} that $\1_D \in H^{2s}_{\text{loc}}(\Omega)$, in the sense that $\psi \1_D \in H^{2s}(\R^d)$ for all $\psi \in \mathcal{C}^\infty_c(\Omega)$. But on the other hand (see \cite[Thm.~11.4]{LioMag72} and \cite[Ch.~33]{Tar07}) functions in $H^{2s}$ cannot contain jump discontinuities as soon as $s\gs1/4$. We might ask ourselves if in the case $s<1/4$ it would be possible to obtain regularity of $\partial D$ from such a source condition. The answer is no, because we would end up with the requirement $A^\ast z \in L^q$ for $q > d/2s$. This is exactly the exponent threshold for $\gamma$ in \eqref{eq:uishoelder}, which implies that $u^\dag$ is H\"older continuous, so $u^\dag = \1_D$ is again not possible.

Let us also mention that in the recent work \cite{NovOno21}, the authors study a denoising scheme with the $W^{s,1}$ seminorm as regularizer, proving preservation of H\"older continuity and hence recreating the well-known result of \cite{CasChaNov11} for total variation denoising. In the fractional case, this seminorm satisifies a coarea formula (first proved in \cite{Vis91}, see also \cite[Thm.~2.2.2]{BucVal16}) in terms of the fractional perimeter, which makes a purely geometric point of view applicable and leads one to expect that results along the lines of those in \cite{ChaDuvPeyPoo17, IglMerSch18, IglMer21} also hold. However, from a numerical point of view the minimization of such a functional is very challenging, since it combines the nonlocality of the fractional formulation and the nonsmoothness arising from being based on $L^1$-type norms.

\subsection{Convergence of subgradients, as appearing in the Dirichlet problem}\label{sec:convsub}

\begin{prop}\label{prop:convergenceoflaplacians}Assume that $A^\ast$ is compact from $\Hm$ to $L^q(\Omega)$, that $u^\dag \in L^p(\Omega)$ is such that $A u^\dag = f$ and $A^\ast z \in \partial_{L^p}\left[\frac12 |\cdot|^2_{H^s} \right](u^\dag)$, and that $\|n\|_\Hm \to 0$ and $\alpha \to 0$ with $\|n\|_\Hm/\alpha \ls C$. Then the minimizers $u_{\alpha, n}$ in \eqref{eq:mingen} satisfy
\begin{equation}\label{eq:convergenceoflaplacians}
\left\|\frac{A^\ast \big( A u_{\alpha,n} - f\big)}{\alpha}+A^\ast z\right\|_{L^q(\Omega)}\to 0 \quad\text{ and }\quad D_{A^\ast z}(u_{\alpha,n},u^\dag) = O(\alpha),
\end{equation}
where $D_{A^\ast z}(u_{\alpha, n}, u^\dag)$ denotes the Bregman distance with respect to the subgradient $A^\ast z \in \partial \left[ \frac12 |\cdot|^2_{H^s}\right](u^\dag)$, that is
\[D_{A^\ast z}(u_{\alpha, n}, u^\dag) := \frac{1}{2}|u_{\alpha, n}|^2_{H^s} - \frac{1}{2}|u^\dag|^2_{H^s} - \scal{A^\ast z}{u_{\alpha, n} - u^\dag}_{(L^q(\Omega), L^p(\Omega))}.\]
\end{prop}
\begin{proof}The proof is based on the one of \cite[Thm.~2]{BurOsh04} (see also \cite{Val21} for more in-depth results in the same fashion). To start, optimality in \eqref{eq:mingen} leads to
\[\|Au_{\alpha, n}-f-n\|^2_{\Hm} + \alpha|u_{\alpha, n}|^2_{H^s} \ls \|n\|^2_{\Hm} + \alpha |u^\dag|^2_{H^s},\]
or \[\frac{1}{2}\|Au_{\alpha, n} - f\|_\Hm^2 +\frac{\alpha}{2} \left( |u_{\alpha, n}|^2_{H^s} - |u^\dag|^2_{H^s}\right) \ls \frac12\|n\|^2_{\Hm}.\]
which using the definitions of $A^\ast$ and of $D_{A^\ast z}$, is equivalent to
\[\frac{1}{2}\|Au_{\alpha, n} - f\|_\Hm^2 + \scal{\alpha z}{Au_{\alpha,n} - f} + \alpha D_{A^\ast z}(u_{\alpha, n}, u^\dag) \ls \frac12 \|n\|^2_{\Hm},\]
or equivalently
\begin{equation}\label{eq:buroshest}\frac{1}{2}\|Au_{\alpha, n} - f + \alpha z\|_\Hm^2 + \alpha D_{A^\ast z}(u_{\alpha, n}, u^\dag) \ls \frac12 \|n\|^2_{\Hm} + \frac{\alpha^2}{2}\|z\|_{\Hm}^2.\end{equation}
Now, classically, dividing by $\alpha$ in \eqref{eq:buroshest} and using $\|n\|_{\Hm}/\alpha \ls C$ gives us the convergence rate $D_{A^\ast z}(u_{\alpha, n}, u^\dag) = O(\alpha)$. We can examine the first term further, obtaining
\[\|Au_{\alpha, n} - f\|_\Hm \ls \big( \|n\|^2_{\Hm} + \alpha^2\|z\|_{\Hm}^2 \big)^{1/2} + \alpha \|z\|_\Hm,\]
which again using the assumption $\|n\|_{\Hm}/\alpha \ls C$ tells us that the family of functions $(Au_{\alpha, n} - f)/\alpha$ is bounded in $\Hm$, so it can be assumed to (up to a subsequence) converge weakly to some element of $\Hm$. Since we have assumed $A^\ast$ to be compact, then the functions $A^\ast(Au_{\alpha, n} - f)/\alpha$ converge strongly in $L^q(\Omega)$ to some element $v_0$. Moreover, since 
\[-\frac{1}{\alpha} A^\ast(Au_{\alpha, n} - f) \in \partial \left[\frac12|\cdot|^2_{H^s}\right](u_{\alpha,n})\text{ and }u_{\alpha, n} \to u^\dag,\]
we can pass to the limit in the weak formulation \eqref{eq:weakforsubg} and use that it has a unique solution, implying not only that we must have $v_0 \in \partial \left[\frac12 |\cdot|^2_{H^s}\right](u^\dag) = \{A^\ast z\}$, but also that the whole sequence converges to $-A^\ast z$.
\end{proof}

It is worthwhile to remark that also for $\TV$ regularization, convergence of subgradients is an important ingredient in the results of geometric convergence of \cite{ChaDuvPeyPoo17, IglMerSch18, IglMer20, IglMer21}, see in particular \cite[Prop.~3]{IglMerSch18}. However, there is one crucial difference: in the total variation case, one does not need to assume $A^\ast$ to be compact, and in fact as soon as the functions $Au_{\alpha,n}-f$ are bounded in $\Hm$ they must converge strongly. One can interpret this in light of the Radon-Riesz property satisfied by Hilbert spaces and uniformly convex Banach spaces, meaning that weak convergence combined with convergence of norms implies strong convergence. For one-homogeneous functionals like $\TV$ the second part is automatically satisfied, since their subgradients are zero-homogeneous (i.e. $\partial\TV (\lambda u)=\partial\TV(u)$). For a quadratic functional like $|\cdot|^2_{H^s}$ this is not the case, so we require compactness in addition.

\begin{remark}\label{rem:yosida}When applying a denoising method and with $n=0$, it is also possible to just consider the properties of resolvents and Yosida regularization, which tell us \cite[Prop.~7.2 (a1) and (d)]{Bre11} that if $f \in H^{2s}(\R^d)=\mathrm{dom}((-\Delta)^s)$ then
\[(-\Delta)^s u_\alpha \to (-\Delta)^s f \text{ strongly in }L^2.\]
\end{remark}

\subsection{Bregman distance in the case of the fractional Laplacian}\label{sec:bregman}
In the present case and assuming $u_{\alpha, n}, u^\dag \in H^{2s}(\R^d) \cap H^s_0(\overline{\Omega})$ additionally to the existence of a source element $z$, the Bregman distance can be computed using \eqref{eq:intbyparts} as
\begin{align*} D_{A^\ast z}(u_{\alpha,n},u^\dag) &= |u_{\alpha,n}|_{H^s}^2 - |u^\dag|_{H^s}^2 - \scal{A^\ast z}{u_{\alpha,n} - u^\dag}_{(L^q(\Omega), L^p(\Omega))} \\
&= |u_{\alpha,n}|_{H^s}^2 - |u^\dag|_{H^s}^2 - \frac{4}{C(d,s)}\scal{(-\Delta)^s u^\dag}{u_{\alpha,n} - u^\dag}_{(L^q(\Omega), L^p(\Omega))} \\
&= |u_{\alpha,n}|_{H^s}^2 - |u^\dag|_{H^s}^2 - \frac{4}{C(d,s)}\scal{(-\Delta)^s u^\dag}{u_{\alpha,n} - u^\dag}_{(L^q(\R^d), L^p(\R^d))} \\
& =|u_{\alpha,n}|_{H^s}^2 - |u^\dag|_{H^s}^2  \\
&\qquad\qquad\;\;\; - \frac{4}{C(d,s)}\scal{(-\Delta)^{s/2} u^\dag}{ (-\Delta)^{s/2} u_{\alpha,n} - (-\Delta)^{s/2} u^\dag}_{L^2(\R^d)} \\
&= \frac{2}{C(d,s)} \bigg( \left\|(-\Delta)^{s/2} u_{\alpha,n}\right\|^2_{L^2} + \left\|(-\Delta)^{s/2} u^\dag\right\|^2_{L^2} \\&\qquad\qquad\;\;\; - 2\scal{(-\Delta)^{s/2} u^\dag}{(-\Delta)^{s/2} u_{\alpha,n}}_{L^2(\R^d)} \bigg) \\
&= \frac{2}{C(d,s)}\left\|(-\Delta)^{s/2} u_{\alpha,n} - (-\Delta)^{s/2} u^\dag\right\|^2_{L^2}=|u_{\alpha,n} - u^\dag|_{H^s}^2,
\end{align*}
and as we saw in Proposition \ref{prop:convergenceoflaplacians} it satisfies
\[D_{A^\ast z}(u_{\alpha,n},u^\dag) = O(\alpha),\]
so that $|u_{\alpha,n} - u^\dag |_{H^s} = O(\alpha^{1/2})$ as well. 

Note that the sequence of equalities above would also hold with the only assumption that $u^\dag$ is in $H^s$, without the need for additional regularity. Indeed, the product
\[\scal{A^\ast z}{u_{\alpha,n}}_{(L^q(\Omega), L^p(\Omega))} \]
can be rewritten using the variational formulation of the equation $A^\ast z = (-\Delta)^s u^\dag$, which combined with \eqref{eq:parseval} would yield the fourth equality.

\subsection{Convergence rates in \texorpdfstring{$L^2$}{L2} norm for denoising}\label{sec:L2rates}
\begin{lemma}\label{lem:L2rates}If $f \in H^s_0(\overline{\Omega})$, $\|n\|_{L^2}/\alpha \ls C$ and $u_{\alpha, n}$ is the unique minimizer of
\begin{equation}\label{eq:fracden}
u \mapsto \int_{\Omega} \big(u(x)-f(x)-n(x)\big)^2 \dd x + \alpha |u|^2_{H^s},
\end{equation}
then we have
\begin{equation}\label{eq:sqrtrate}\|u_{\alpha, n} - f\|_{L^2} \ls C(f) \alpha^{1/2}.\end{equation}
\end{lemma}
\begin{proof}
Since $f\in H^s(\R^d)$ we can test minimality in \eqref{eq:fracden} with $f$ to get
\begin{equation}\label{eq:testbyf}\big(\|u_{\alpha, n} - f\|_{L^2} - \|n\|_{L^2}\big)^2 \ls  \|u_{\alpha, n} - f - n\|_{L^2}^2  + \alpha |u_{\alpha, n}|_{H^s}^2 \ls \|n\|^2_{L^2} + \alpha |f|_{H^s}^2,\end{equation}
and again testing minimality with zero we also get
\[\big(\|u_{\alpha, n}\|_{L^2} - \|f+n\|_{L^2}\big)^2 \ls\|u_{\alpha, n} - f - n\|_{L^2}^2  + \alpha |u_{\alpha, n}|_{H^s}^2 \ls \|f+n\|^2_{L^2},\]
%or
%\[\|u_{\alpha, n}\|^2_{L^2} \ls\|u_{\alpha, n} - f - n\|_{L^2}^2  + \alpha |u_{\alpha, n}|_{H^s}^2 \ls 2\|u_{\alpha, n}\|_{L^2}\|f+n\|_{L^2},\]
which implies
\[\|u_{\alpha, n}\|_{L^2} \ls 2\|f+n\|_{L^2}.\]
Using this last inequality in \eqref{eq:testbyf} and the parameter choice, we end up with
\begin{align*}\|u_{\alpha, n} - f\|_{L^2}^2 &\ls \alpha |f|_{H^s}^2 + 2\|u_{\alpha, n} - f\|_{L^2}\|n\|_{L^2} + \|n\|^2_{L^2} \\&\ls \alpha |f|_{H^s}^2 + \big(2\|f\|_{L^2}+4\|f+n\|_{L^2}\big)\|n\|_{L^2} + \|n\|^2_{L^2} \\ & \ls\alpha \big(|f|_{H^s}^2 + 2C\|f\|_{L^2}+4C\|f+n\|_{L^2}\big) + C^2\alpha^2,\end{align*}
which implies \eqref{eq:sqrtrate}.
% which implies, testing minimality again with zero that
% \begin{align*}\|u_{\alpha, n} - f\|_{L^2}^2 &\ls \alpha |f|_{H^s}^2 + \|u_{\alpha, n} - f\|_{L^2}\|n\|_{L^2} \\&\ls \alpha |f|_{H^s}^2 + \big(\|f\|_{L^2}+\|f+n\|_{L^2}\big)\|n\|_{L^2} \\ & \ls\alpha \big(|f|_{H^s}^2 + C^{-1}\|f\|_{L^2}+C^{-1}\|f+n\|_{L^2}\big),\end{align*}
% which is \eqref{eq:sqrtrate}.
\end{proof}

\begin{remark}The lemma above applies to $f=\1_D$ when $s<1/2$ and $D$ of finite perimeter, since then \[|\1_D|^2_{H^s}=\per_{2s}(D) \ls \per(D) < + \infty.\]
In contrast, as noted in Section \ref{sec:fractv}, when $s \geq 1/2$ functions in $H^s$ contain no jump discontinuities.
\end{remark}

\section{Geometric convergence for denoising with piecewise constant ideal data}\label{sec:indicatrix}
We want to use the barriers constructed in \cite{SavVal11,SavVal14} to obtain geometric convergence for the minimizer of 
\begin{equation}
\argmin_{u \in H^s_0(\overline{\Omega})} \int_\Omega \big( u(x) - f(x) - n(x) \big)^2 + \alpha |u|^2_{H^s},
 \label{eq:fracrofW}
\end{equation}
as $\alpha$ and $n$ tend to zero simultaneously. In this section we restrict ourselves to $f = \1_D$, the characteristic function of an open set $D \subset \Omega$. This implies in particular that if the noise $n$ vanishes, minimizers will have values in between $0$ and $1$, by the following lemma:

\begin{lemma}[Truncation]\label{lem:threshold}
Assume that $n=0$ and $f \in L^\infty(\Omega)$ is such that $\essinf_\Omega f \ls 0 \ls \esssup_\Omega f$. Then the minimizer $u_{\alpha, 0}$ of \eqref{eq:fracrofW} satisfies
\[u_{\alpha, 0}(x) \in \left[\essinf_{\Omega} f, \esssup_{\Omega} f\right] \text{ for a.e. } x \in \Omega.\]
\end{lemma}
\begin{proof}Let $u$ be arbitrary and denote a truncation at level $T:=\esssup_{\Omega} f$ of it by $u^T := \min(u, \esssup_{\Omega} f )$. First, we have that 
\[\int_{\Omega} \big(u^T(x)-f(x)\big)^2 \dd x \ls \int_{\Omega} \big(u(x)-f(x)\big)^2 \dd x.\]
Moreover, thanks to \eqref{eq:complementformula} the seminorm $|u|_{H^s}$ can be written as the sum of two interactions. For the first term we have
\begin{align*}&\int_{\Omega}\int_{\Omega} \frac{|u^T(x) - u^T(y)|^2}{|x-y|^{d+2s}} \dd x\dd y\\ 
&\qquad=\int_{\{u\gs T\}}\int_{\{u< T\}} \frac{|T - u(y)|^2}{|x-y|^{d+2s}} \dd x\dd y + 
\int_{\{u< T\}}\int_{\{u\gs T\}} \frac{|u(x) - T|^2}{|x-y|^{d+2s}} \dd x\dd y \\&\qquad\qquad+ \int_{\{u<T\}}\int_{\{u < T\}} \frac{|u(x) - u(y)|^2}{|x-y|^{d+2s}} \dd x\dd y,
\end{align*} 
which leads to 
\begin{align*}&\int_{\Omega}\int_{\Omega} \frac{|u^T(x) - u^T(y)|^2}{|x-y|^{d+2s}} \dd x\dd y\\ 
&\qquad\ls \int_{\{u\gs T\}}\int_{\{u< T\}} \frac{|u(x) - u(y)|^2}{|x-y|^{d+2s}} \dd x\dd y + 
\int_{\{u< T\}}\int_{\{u\gs T\}} \frac{|u(x) - u(y)|^2}{|x-y|^{d+2s}} \dd x\dd y \\&\qquad\qquad+ \int_{\{u<T\}}\int_{\{u < T\}} \frac{|u(x) - u(y)|^2}{|x-y|^{d+2s}} \dd x\dd y \\
&\qquad\ls \int_{\Omega}\int_{\Omega} \frac{|u(x) - u(y)|^2}{|x-y|^{d+2s}} \dd x\dd y,
\end{align*}
and for the second term, owing to the assumption $T \gs 0$, that 
\begin{align*}&\int_{\Omega^c}\int_{\Omega} \frac{|u^T(x)|^2}{|x-y|^{d+2s}} \dd x\dd y\\
&\qquad=\int_{\Omega^c}\int_{\{u>T\}} \frac{T^2}{|x-y|^{d+2s}} \dd x\dd y + \int_{\Omega^c}\int_{\{u \ls T\}} \frac{|u(x)|^2}{|x-y|^{d+2s}} \dd x\dd y \\
&\qquad\ls \int_{\Omega^c} \int_{\Omega} \frac{|u(x)|^2}{|x-y|^{d+2s}} \dd x\dd y.
\end{align*}
Combined, these inequalities imply that $|u^T|_{H^s} \ls |u|_{H^s}$. For $\max (u, \essinf_{\Omega} f)$ one proceeds analogously. By uniqueness of the minimizer of \eqref{eq:fracrofW} (with $n=0$) this implies $u^T=u$.
\end{proof}

\begin{remark}
The assumption $\essinf_\Omega f \ls 0 \ls \esssup_\Omega f$ is not superfluous, since we have a homogeneous Dirichlet boundary/outer condition. This implies that the denoising procedure does not preserve constants and has a bias towards zero, since the second term of \eqref{eq:complementformula} imposes a decay as $d(x, \partial \Omega) \to 0$.
\end{remark}

Next we explore the behaviour of the energy with respect to rescaling in space, which results in a scaling factor multiplying the right hand side of the Euler-Lagrange equation. This fact will be used to analyze the behaviour of solutions using a `fixed amount of diffusion'.

\begin{lemma}[Scaling]\label{lem:scaling}
If $n=0$ and $u_\alpha$ is the minimizer of \eqref{eq:fracrofW}, then the function
\[u_\alpha^{\rho}(x):=u_\alpha(\rho x)\]
minimizes
\begin{equation}
\int_{\rho^{-1} \Omega} \big(u(\hat x) - f(\rho \hat x)\big)^2 \dd \hat x + \alpha \rho^{-2s}|u|^2_{H^s}
 \label{eq:scaledfracrof}
\end{equation}
\end{lemma}
\begin{proof}
Just notice that 
\[ \int_\Omega (u_\alpha-f)^2 = \int_{\Omega} \big(u^\rho_\alpha(x / \rho) - f(x)\big)^2 \dd x = \rho^{d} \int_{\rho^{-1} \Omega} \big(u^\rho_\alpha(\hat x) - f(\rho \hat x)\big)^2 \dd \hat x\]
and
\begin{align*}|u|^2_{H^s} &= \int_\Omega \int_\Omega \frac{|u^\rho_\alpha\left(x/\rho\right)-u^\rho_\alpha\left(y/\rho\right)|^2}{|x-y|^{d+2s}} \dd x \dd y + 2\int_\Omega \int_{\R^d \setminus \Omega} \frac{|u^\rho_\alpha\left( x/\rho\right)-u^\rho_\alpha\left(y/\rho\right)|^2}{|x-y|^{d+2s}} \dd x \dd y \\
& = \rho^{2d} \int_{\rho^{-1}\Omega} \int_{\rho^{-1}\Omega} \frac{|u^\rho_\alpha\left(\hat x\right)-u^\rho_\alpha\left(\hat y\right)|^2}{|\rho \hat x-\rho \hat y|^{d+2s}} \dd \hat{x} \dd \hat{y} \\&\qquad+ 2\rho^{2d}\int_{\rho^{-1}\Omega} \int_{\R^d \setminus \rho^{-1}\Omega} \frac{|u^\rho_\alpha\left(\hat x\right)-u^\rho_\alpha\left(\hat y\right)|^2}{|\rho \hat x-\rho \hat y|^{d+2s}} \dd \hat{x} \dd \hat{y}\\
& = \rho^{d-2s} |u^\rho_\alpha|^2_{H^s}.\qedhere
\end{align*}
\end{proof}

\begin{remark}\label{rem:scalingone}
From Lemma \ref{lem:scaling}, we see that to have a fixed factor $1$ on the $|\cdot|^2_{H^s}$ term we should consider $\rho = \alpha^{1/2s}$, which when $\alpha \to 0$ corresponds to denoising `zoomed in' versions $f^\rho$ of $f$ defined on $\rho^{-1} \Omega \nearrow \R^d$ . In that case for the transformation $\hat x = x/\rho$ a ball of radius $R$ in $\rho^{-1} \Omega$ corresponds to a ball of radius $R\alpha^{1/2s}$ in the original domain $\Omega$.
\end{remark}

To prove geometric properties a family of barriers is needed, which is introduced in \cite[Lemma 2]{SavVal11}:
\begin{lemma}[Barrier]
\label{lem:barrier}
There exists a constant $R_0(d,s)>0$ such that for any \[R\gs R_0(d, s)\] and $b>a$ there exists a constant $C_B:=C_B(d,s,b-a)$, nondecreasing in $b-a$, and a rotationally symmetric function $w \in \C(\R^d, [a+C_B R^{-2s},b])$ which satisfies
\begin{equation}\label{eq:barrierfar}w=b \text{ on }\R^d \setminus B(0,R),\end{equation}
\begin{equation}\label{eq:barrierlapbound}-(-\Delta)^s w(x) = \int_{\R^d} \frac{w(x)-w(y)}{|x-y|^{n+2s}} \dd y \ls w(x)-a\end{equation}
and finally
\begin{equation}\label{eq:barriervalbound}\frac{1}{C_B} \left(R+1-|x| \right)^{-2s} \ls w(x)-a \ls C_B \left(R+1-|x| \right)^{-2s}\end{equation}
for every $x \in B(0,R)$.
\end{lemma}
\begin{proof}
The barrier constructed in \cite[Lemma 2]{SavVal11} corresponds to $a = -1$ and $b=1$, and in their case a parameter $\tau$ appears multiplying the right hand side of \eqref{eq:barrierlapbound}, which in our case is always $\tau = 1$. Then, we just notice that multiplying $w$ by a constant changes the offset $a$ in \eqref{eq:barrierlapbound} but not the terms with $w$ by homogeneity, while in \eqref{eq:barriervalbound} it just affects the constant $C_B$.
\end{proof}

Our strategy is to rescale using Lemma \ref{lem:scaling} as indicated in Remark \ref{rem:scalingone}, to then consider comparison with barriers only in situations with $\alpha=1$. In that case we have:

\begin{lemma}[Avoidance]\label{lem:compalphaone}
Let $\D \subset \Om \subset \R^d$ with $\Om$ open and bounded, and $\u_n$ be the unique minimizer of 
\begin{equation}
u \mapsto \int_{\Om} \big(u(x) - \1_{\D}(x)-n(x)\big)^2 \dd x + |u|^2_{H^s(\Om)}.
\label{eq:onefracrof}
\end{equation}
Then for every $R>R_0:=R_0(d, s)$ (following the notation of Lemma \ref{lem:barrier}), radius $r<R$, level $\otheta > 0$ and tolerance $\eta < \otheta$ satisfying 
\begin{equation}\label{eq:boundcondition}C_B(R+1-r)^{-2s} \ls \otheta-\eta\end{equation} 
where $C_B:=C_B(d,s,1)$, we have that whenever 
\begin{equation}\label{eq:noisetolerance}\|\u_n - \u_0\|_{L^\infty} \ls \eta,\end{equation}
also 
\begin{equation}\label{eq:levelsetnottthere}\begin{gathered}\u_n \ls \otheta\text{ on }B(x,r)\text{ and }\u_n(x) < \otheta,\\\text{ for all }x \in \Om \setminus \D \text{ with }\max\big( d(x, \partial \D), d(x, \partial \Om)\big)>R.\end{gathered}\end{equation}
\end{lemma}
\begin{proof}
Let $x \in \Om \setminus \D$ with $\max\big( d(x, \partial \D), d(x, \partial \Om)\big) >R$, and $w$ be the translation to $x$ of the barrier of Lemma \ref{lem:barrier} with $a = 0$ and $b=1$. Moreover, assume that $n=0$ for now. Then we have
\[-(-\Delta)^s w \ls w  \text{ on } B(x,R)\]
and since $d(x, \partial \Om) > R$ and $d(x, \partial \D) > R$ also the Euler Lagrange equation
\[ (-\Delta)^s \u_0 = \1_\D-\u_0 = -\u_0 \text{ on } B(x,R). \]
Summing up these two identities, we obtain
\[ (-\Delta)^s (\u_0-w) +\u_0- w \ls 0 \text{ on } B(x, R),\]
which allows (extending $\u_0$ by $0$ outside $\Om$) to apply the maximum principle of Proposition \ref{prop:maxpple} to $\u_0-w$ on $B(x,R)$, and taking into account that by Lemma \ref{lem:threshold} we have $\u_0-w \ls 0$ on $\R^d \setminus B(x,R)$, to finally obtain
\[\u_0(x) \ls w(x) \text{ for all }x \in B(x,R),\]
where we could conclude for all $x$ and not only a.e.~because the regularity results of Section \ref{sec:regularity} are applicable and $\u_0$ is in fact (locally) H\"older continuous. To conclude, notice that by \eqref{eq:noisetolerance} we have, again pointwise since $\u_n$ is also continuous, that
\[\u_n \ls w+\eta \text{ on }B(x,R),\]
but then \eqref{eq:boundcondition} and \eqref{eq:barriervalbound} imply
\begin{equation}\label{eq:wboundwitheta}w \ls C_B(R+1-r)^{-2s} \ls \otheta - \eta \text{ on }B(x,r) \subset B(x,R),\end{equation}
so that we get
\[B(x,r) \cap \{\u_n > \otheta\} = \emptyset,\]
as desired. To see that $\u_n(x) < \otheta$, just notice that the expression in \eqref{eq:wboundwitheta} is strictly increasing in $r$.
\end{proof}

\begin{remark}
Note that since there is a fixed universal lower bound for $R$ and also the tolerance $\eta$ needs to be accomodated in \eqref{eq:boundcondition}, a significant distance from $x$ to $\partial \hat{D}$ and $\partial \hat{\Omega}$ is needed. This will be the case when we apply this lemma below, because the domains $\hat{D}$ and $\hat{\Omega}$ will be rescaled versions of fixed ones, with freedom to choose the rescaling factor.
\end{remark}

We would also like to have a conclusion, for points in $D$, of the type $\u_n(x) > \otheta$. In this case, the Dirichlet condition on $\R^d \setminus \Om$ makes this harder and harder to satisfy if $x$ is very close to $\partial \Om$. To take this effect into account, we use the input $1-\1_D$ and obtain a lower bound on the corresponding output. This leads us to consider denoising of the constant function $1$. Again because of the Dirichlet condition on $\R^d \setminus \Om$, the resulting solution is not constant, and instead it must decay close to the boundary $\partial \Om$, so let us denote this result by $\Xih$. With it, we have then:

\begin{lemma}\label{lem:compalphaoneD}
In the situation of Lemma \ref{lem:compalphaone} but replacing \eqref{eq:boundcondition} by
\begin{equation}\label{eq:flippedboundcondition}C_B(R+1-r)^{-2s} \ls 1-\otheta-\eta,\end{equation} 
we also have that
\begin{equation}\label{eq:levelsetnottthereD}\begin{gathered}\u_n \gs \otheta - (1-\Xih) \text{ on }B(x,r)\text{ and }\u_n(x) > \otheta - \big(1-\Xih(x)\big), \\ \text{ for all }x \in \D\text{ with }d(x, \partial \D)>R.\end{gathered}\end{equation}
\end{lemma}
\begin{proof}
We consider as announced the denoising of $1-\1_{\D}$ with solution $\Xih - \u_{0}$ and choose a point $x \in \D$ with $d(x,\R^d \setminus \D) >R$, where we note that since $\D \subset \Om$, it is automatically true that $d(x,\partial \Om) >R$. 

Let $w$ again be the barrier of Lemma \ref{lem:barrier} with $a=0$ and $b=1$. We have then
\[-(-\Delta)^s w \ls w \text{ and }(-\Delta)^s (\Xih-\u_0) = 1-\1_D-(\Xih-\u_0) = -(\Xih-\u_0) \text{ on } B(x,R)\]
Summing these two equations, we obtain
\[(-\Delta)^s (\Xih-\u_0-w) +(\Xih-\u_0-w) \ls 0 \text{ on } B(x,R). \]
and moreover, since $\Xih, u_0 \in [0,1]$ which implies $\Xih-\u_0-w \ls 0$ on $\R^d \setminus B(x,R)$, we have
\[(\Xih-\u_0-w) \ls 0 \text{ on }\R^d \setminus B(x,R),\]
after extending $\Xih, u_0$ by $0$ outside $\Omega$. The maximum principle of Proposition \ref{prop:maxpple} implies then that the same inequality holds also in $B(x,R)$. The rest follows exactly as in Lemma \ref{lem:compalphaone}, taking into account that the level gets transformed to $1-\otheta$ which is the distance to the new energy well (where the fidelity term vanishes).
\end{proof}

We also use the barriers to quantify the boundary effects on $\Xih$:
\begin{lemma}\label{lem:ctdecay}
Let $\Xih$ be the unique minimizer of 
\begin{equation}
u \mapsto \int_{\Om} (u(x) - 1)^2 \dd x + |u|^2_{H^s(\Om)}.
\end{equation}
Then we have that for all $x$ with $d\big(x, \partial \Om\big) > R_0$,
\begin{equation}\label{eq:decay}1-\Xih(x) \ls C_B \Big(1+d\big(x, \partial \Om\big)\Big)^{-2s}.\end{equation}
\end{lemma}
\begin{proof} Using Lemma \ref{lem:threshold} we know that $\Xih$ takes values in $[0, 1]$. For convenience and using the linearity of the equations, we consider $\zeta = - \Xih$ and the barrier of Lemma \ref{lem:barrier} with $a=-1$ and $b = 0$. In that case setting $R:=d\big(x, \partial \Om\big)$ we have
\begin{alignat*}{3}
(-\Delta)^s \zeta &= -1-\zeta &&\text{ on }B(x,R) \text{ and }\\
-(-\Delta^s) w &\ls w+1 &&\text{ on }B(x,R),
\end{alignat*}
which added up and considering \eqref{eq:barrierfar} imply 
\[\begin{cases}(-\Delta)^s (\zeta - w)+(\zeta - w) \ls 0 &\text{ on }B(x,R) \\ (\zeta - w) \ls 0 &\text{ on }\R^d \setminus B(x,R).\end{cases}\]
By the maximum principle of Proposition \ref{prop:maxpple} and \eqref{eq:barriervalbound} we get then
\begin{align*}\zeta(x) &\ls w(x) \ls -1 + C_B(1+R)^{-2s}, \text{ or }\\\Xih(x) &\gs 1 - C_B(1+R)^{-2s}\end{align*}
and hence \eqref{eq:decay}.
\end{proof}

Undoing the rescaling in Lemmas \ref{lem:compalphaone}, \ref{lem:compalphaoneD} and \ref{lem:ctdecay}, we arrive at:
\begin{theorem}\label{thm:hausmaxprinc}Let $\Omega$ be a bounded domain with Lipschitz boundary and satisfying an exterior ball condition, $D\subset \Omega$ also with Lipschitz boundary and 
\begin{equation}\label{eq:donottouch}\inf_{x \in \partial D} d(x, \partial \Omega) > 0,\end{equation}
where we remark that $\partial D$ is considered in $\R^d$. Moreover, let $u_{\alpha, n}$ be the unique minimizer of 
\begin{equation}\label{eq:fracrof}
u \mapsto\int_{\Omega} \big(u(x) - \1_{D}(x)-n(x)\big)^2 \dd x + \alpha |u|^2_{H^s}.
\end{equation}
Then, if $\|n\|_\Hm \to 0$, $\alpha \to 0$ and the parameter choice is such that
\begin{equation}\label{eq:paramchoice}\frac{\|n\|_{L^q(\Omega)}}{\alpha} \to 0 \text{ for some } q > \frac{d}{2s},\end{equation}
then we have that for almost every $\theta \in (0,1)$
\[d_H(\partial \{u_{\alpha, n} > \theta\}, \partial D) \xrightarrow[\alpha \to 0]{} 0.\]
\end{theorem}
\begin{proof}
To maintain the strategy of Lemma \ref{lem:compalphaone} of using the maximum principle and the barrier $w$ but with some noise $n$ added to $f$, we will need to control the effect of the noise in $L^\infty$ norm. Denoting $u_{\alpha, n}$ and $u_{\alpha, 0}$ the corresponding solutions with and without noise, we have that
\begin{equation}\begin{cases}(-\Delta)^s (u_{\alpha, n} - u_{\alpha, 0}) + (u_{\alpha, n} - u_{\alpha, 0})/\alpha = n/\alpha &\text{ on }\Omega \\u_{\alpha, n}-u_{\alpha, 0} = 0 &\text{ on }\R^d \setminus \Omega\end{cases},\end{equation}
and on this equation we may apply the boundedness estimate of Proposition \ref{prop:boundedsol} to obtain
\[\|u_{\alpha, n} - u_{\alpha, 0}\|_{L^\infty(\Omega)} \ls C_S(\Omega, q, d, s) \left(\frac{\|n\|_{L^p(\Omega)}}{\alpha}\right)^{2d/(d-2s)} \text{ for each } q > \frac{d}{2s}.\]
This means that with the parameter choice we are able to enforce, for any $\eta >0$ and possibly looking further into the vanishing sequence of $\alpha$ and noise instances that
\[\|u_{\alpha, n} - u_{\alpha, 0}\|_{L^\infty(\Omega)} \ls \eta.\]
Now, we can apply Lemma \ref{lem:compalphaone} on the rescaled domain $\Om = \alpha^{-1/2s} \Omega$ with parameters (independent of $\alpha$) given by $\otheta = \theta$, $\eta = \theta/2$, and $r=R/2$ for a radius $R>0$ satisfying
\begin{equation}\label{eq:Rbound1}R>\max\big(C\theta^{-1/2s}-C, R_0\big)\end{equation}
for an adequate constant $C$ that could be made explicit by solving in \eqref{eq:boundcondition} with these parameters. The conclusion of Lemma \ref{lem:compalphaone}, taking into account Remark \ref{rem:scalingone} after Lemma \ref{lem:scaling} leads us then to
\[u_{\alpha, n}(x) \ls \theta \text{ for all }x \in D^c\text{ with }\max\big( d(x, \partial D), d(x, \partial \Omega)\big)>R\alpha^{1/2s},\]
from which we would like to deduce that
\[\sup_{x \in \{u_{\alpha, n} > \theta \}} d(x, D) \ls R \alpha^{1/2s} \text{ for }\alpha\text{ small enough},\]
which is one half of the definition \eqref{eq:hausdorffdist} of $d_H(\{u_{\alpha, n} > \theta \}, D)$. We cannot immediately conclude this though, since we need to additionally ensure that
\[\{u_{\alpha, n} > \theta \} \cap \{d(x,\partial \Omega) \ls R\alpha^{1/2s}\} = \emptyset \text{ for }\alpha\text{ small}.\]
For this, we can first denote by $\Xi_\alpha$ the result of denoising with input the constant function $1$, no noise and the given regularization parameter, and then use the comparison principle and the result of regularity up to the boundary of Proposition \ref{prop:boundaryregularity}, to obtain
\[\theta < u_{\alpha,n}(x) \ls (1+\eta)\,\Xi_\alpha(x) \ls (1+\eta) \big(d(x,\partial \Omega)\big)^s \|\Xi_\alpha\|_{\mathcal{C}^{0,s}(\overline{\Omega})} \ls \frac{C_G(1+\eta)}{\alpha} \big(d(x,\partial \Omega)\big)^s,\]
so that taking into account $\eta = \theta/2$, we have for all $x$ for which $u_{\alpha,n}(x) >\theta$ that
\begin{equation}\label{eq:spanishbound}d(x,\partial \Omega)> \left( \frac{2\theta}{2+\theta}\frac{1}{C_G}\alpha\right)^{1/s} > R\alpha^{1/2s},\end{equation}
for all $0<\alpha \ls \alpha_\theta$ and some $\alpha_\theta$ depending only on $\theta$. Note that $R$ depends itself on $\theta$, but still $\alpha_\theta > 0$ for all $\theta$, although $\alpha_\theta \to 0$ as $\theta \to 0$.

For the other half, we need a conclusion of the type $u_{\alpha, n}(x) > \theta$ for $x \in D$. Under the same rescaling onto $\Om = \alpha^{-1/2s} \Omega$, the function $\Xi_\alpha$ defined above transforms to $\Xih$ appearing in Lemmas \ref{lem:compalphaoneD} and \ref{lem:ctdecay}. From the latter we find then that for all $x \in D$
\begin{equation}\label{eq:Xibound}\begin{aligned}1-\Xi_\alpha(x)&\ls C_B\big(1+d(x,\partial \Omega)\alpha^{-1/2s}\big)^{-2s} \\&\ls C_B\big(1+\inf_{x \in D} d(x,\partial \Omega)\alpha^{-1/2s}\big)^{-2s} := B_D(\alpha) \xrightarrow[\alpha \to 0]{} 0,\end{aligned}\end{equation}
where we have used \eqref{eq:donottouch}. Since this uniform bound vanishes as $\alpha \to 0$, we can always restrict $\alpha$ so that $B_D(\alpha) < (1-\theta)/3$. We can then use this estimate to apply Lemma \ref{lem:compalphaoneD} with $\otheta=\theta+(1-\theta)/3$ (which is always stricly below $1$) and $\eta=(1-\theta)/3$, and imposing on $R$ the additional condition
\begin{equation}\label{eq:Rbound2}R>\max\big(C'(1-\theta)^{-1/2s}-C',R_0\big),\end{equation}
where we remark that the constant $C'$, derived from \eqref{eq:flippedboundcondition}, is independent of $\alpha$ but different than the one appearing in \eqref{eq:Rbound1}. As conclusion of the lemma, we finally get that 
\[u_{\alpha, n}(x) > \theta \text{ for all }x \in D,\]
and this in turns means that
\[\sup_{x \in \{u_{\alpha, n} \ls \theta \}} d\big(x, D^c\big) \ls R \alpha^{1/2s} \text{ for all }\alpha \text{ such that }B_D(\alpha)<(1-\theta)/3.\]
Since $\{u_{\alpha, n} \ls \theta \} = \{u_{\alpha, n} > \theta \}^c$, we can combine the two estimates above for a distance between boundaries (see \cite[Prop.~2.6]{IglMer21}, for example) and end up with
\[\sup_{x \in \partial \{u_{\alpha, n} > \theta \}} d(x, \partial D) \to 0.\]
To complete the Hausdorff convergence $d_H(\partial \{u_{\alpha, n} > \theta \}, \partial D) \to 0$ we also need that
\[\sup_{x \in \partial D} d\big(x, \partial \{u_{\alpha, n} > \theta \}\big) \to 0.\]
This is satisfied (see \cite[Props.~2.2 and 2.6]{IglMer21} for a proof) if $\{u_{\alpha, n} > \theta \}$ converges to $D$ in $L^1$ (which is true for a.e.~$\theta$, since $u_{\alpha,n} \to \1_D$ in $L^2$) and $D$ satisfies inner and outer density estimates as in \eqref{eq:densityestA}. These density estimates are in particular implied by $\partial D$ being either a strong Lipschitz boundary or a Lipschitz manifold.
\end{proof}

\begin{remark}
We note that in these arguments we needed to assume the conditions \eqref{eq:Rbound1} and \eqref{eq:Rbound2} in which $R\to \infty$ if $\theta \to 0$ or $\theta \to 1$. A restriction in $R$ from below in turn forces us to choose $\alpha$ smaller and smaller, so the speed of convergence one can obtain degenerates if $\theta$ is very close to the values $0$ and $1$ attained by $\1_D$. This motivates the results of Section \ref{sec:rates} below.
\end{remark}

\begin{remark}\label{rem:reallaplacianconv}
The proof of the result above relies just on comparison principles and the barriers of Lemma \ref{lem:barrier}. If instead of fractional Laplacian regularization one would use the $H^1$ seminorm, leading to the usual Laplacian, one could use the same proof just by replacing the barriers with the ones of \cite{CafCor95} (see the proof of Lemma 3 there). In the total variation case this kind of result with denoising and $\1_D$ as limit is true even without regularity assumptions on $D$, see \cite[Thm.~1.2]{IglMer21}.
\end{remark}

Such as a result is also straightforward to extend to piecewise constant functions on regular enough partitions:

\begin{cor}\label{cor:pwconst}
Assume that $f \gs 0$ is piecewise constant and with compact support on $\Omega$, that is, there are $0 =c_0 < c_1 < \ldots < c_N$ and $\Omega \supset \Omega_1 \supset \ldots \supset \Omega_N$ with $\inf_{x \in \Omega_1}  d(x, \partial \Omega)>0$ such that
\begin{equation}\label{eq:pwcf} f=\sum_{i=1}^N \big( c_i-c_{i-1} \big) \1_{\Omega_i}, \text{ so that }\Omega_i = \{ f > c_{i-1} \},\end{equation}
where the boundaries $\partial \Omega_i$ are Lipschitz, and $n, \alpha$ are such that the parameter choice condition \eqref{eq:paramchoice} holds. Then for $u_{\alpha, n}$ the minimizers of \eqref{eq:mingen} and $\theta \in (c_{i-1}, c_i)$ with $i=1, \ldots, N$ we have that
\[d_H(\partial \{u_{\alpha, n} > \theta\}, \partial \Omega_i) \xrightarrow[\alpha \to 0]{} 0.\]
\end{cor}
\begin{proof}
The proof follows in a completely analogous manner, when considering $x \in \Omega_i \setminus \Omega_{i-1}$ with $\max\big(d(x, \partial \Omega_{i-1}), d(x, \partial \Omega_i)\big) > R \alpha^{1/2s}$. The barrier should then have $a = f(x) = c_i$ and be above $u_{\alpha, 0}$ far away from $x$, which is accomplished with $b = c_N = \max_\Omega f$.
\end{proof}

\begin{remark}
In Corollary \ref{cor:pwconst} we have used that every level set of $f$ should have a Lipschitz boundary. Note that this is a stronger assumption than the domains of constancy $\Omega_i \setminus \Omega_{i-1}$ having a Lipschitz boundary. One can easily construct a counterexample with the union of two tangent balls in the plane, for instance.
\end{remark}

\subsection{Convergence rates in Hausdorff distance for denoising with indicatrix ideal data}\label{sec:rates}
Since we have obtained our results by comparison with barriers with values between those attained by the piecewise constant ideal input $f$, the speed of the convergence one can obtain with this method depends on how close the level $\theta$ is to those values. This is reflected in the following convergence rates result:
\begin{theorem}\label{thm:convrates}In the situation of Theorem \ref{thm:hausmaxprinc} and with $s<1/2$, there is a constant $\alpha_0(\Omega, D, \theta)>0$ for which we have for $\theta \in (0,1)$ the estimate
\begin{equation}\label{eq:hausrate}
d_H\big(\partial \{u_{\alpha, n} > \theta\}, \partial D\big) \ls C(\Omega, D)\max\left(\frac{1}{\theta},\ \frac{1}{(1-\theta)}\right)^{\min(2/d,\,1/2s)}\alpha^{1/d},
\end{equation}
for all $\alpha \ls \alpha_0(\Omega, D, \theta)$.
\end{theorem}
\begin{proof}
Using the Markov/Chebyshev inequality we get for $t\gs 0$ and any $f$ that,
\[|\{u_{\alpha, n} - f \gs t\}| \ls \frac{1}{t^2} \int_{\{(u_{\alpha, n}-f) \gs t\}} |u_{\alpha, n}-f|^2 \ls \frac{1}{t^2} \Vert u_{\alpha, n}-f\Vert_{L^2}^2. \]
For $t \ls 0$, we have as well
\[|\{u_{\alpha, n} - f \ls t\}| \ls \frac{1}{t^2} \int_{\{(u_{\alpha, n}-f)\ls t\}} |u_{\alpha, n}-f|^2 \ls \frac{1}{t^2} \Vert u_{\alpha, n}-f\Vert_{L^2}^2. \]
Now, since we assume that $f$ is an indicatrix $f=\1_D$, one can rewrite, if $t \gs 0$,
\[| \{u_{\alpha, n} > t+1\} \cap D| + |\{u_{\alpha, n} > t\} \setminus D| \ls \frac{1}{t^2} \Vert u_{\alpha, n}-\1_D\Vert_{L^2}^2\]
which in particular implies for $\theta > 0$ that
\[| \{u_{\alpha, n} > \theta\} \setminus D| \ls \frac{1}{\theta^2} \Vert u_{\alpha, n}-\1_D\Vert_{L^2}^2.\]
% If $t\ls 0$, one obtains
% \[| \{u_{\alpha, n} < t+1\} \cap D| + |\{u_{\alpha, n} < t\} \setminus D| \ls \frac{1}{t^2} \Vert u_{\alpha, n}-f\Vert_{L^2}^2,\]
% which implies that for all $\theta<0$,
% \[| \{u_{\alpha, n} < \theta\} \setminus D| \ls \frac{1}{\theta^2} \Vert u_{\alpha, n}-f\Vert_{L^2}^2.\]
For $t \in (-1,0)$, we obtain similarly
\[|D \setminus \{u_{\alpha, n} > t+1\}| \ls |\{u_{\alpha, n} - \1_D \ls t\}| \ls \frac{1}{t^2} \Vert u_{\alpha, n}-\1_D\Vert_{L^2}^2\]
which also reads, for $\theta = (t+1) \in (0,1)$,
\[|D \setminus \{u_{\alpha, n} > \theta \}| \ls \frac{1}{(\theta-1)^2} \Vert u_{\alpha, n}-\1_D\Vert_{L^2}^2.\]
Collecting these results, we obtain for all $\theta \in (0,1)$ that
\[ |D \Delta \{u_{\alpha, n} > \theta\} | \ls \left( \frac{1}{\theta^2} + \frac{1}{(\theta-1)^2} \right) \Vert u_{\alpha, n}-\1_D\Vert_{L^2}^2.\] 
Now, if $D$ satisfies the density estimates
\begin{equation}\label{eq:densityestD}\frac{|D \cap B(x,r)|}{|B(x,r)|}\gs C_D \text{ and }\frac{|B(x,r) \setminus D|}{|B(x,r)|}\gs C_D \text{ for all }r\ls r_0 \text{ and } x \in \partial D,\end{equation}
we infer that
\[|D \Delta \{u_{\alpha, n} > \theta\} | \gs C_D |B(0,1)| r^d \text{ whenever } r \ls \min \left( \sup_{x \in \partial D} d\big(x, \partial \{u_{\alpha, n} > \theta \}\big), r_0 \right),\]
but since $|D \Delta \{u_{\alpha, n} > \theta\} | \to 0$, for $\alpha$ small enough we end up for each $\theta \in (0,1)$ with
\begin{align*}\sup_{x \in \partial D} d\big(x, \partial \{u_{\alpha, n} > \theta \}\big) &\ls (C_D|B(0,1)|)^{-1/d} |D \Delta \{u_{\alpha, n} > \theta\}|^{1/d} \\&\ls \left(\frac{2}{C_D}\right)^{1/d}\max\left( \frac{1}{\theta^{2/d}},\ \frac{1}{(1-\theta)^{2/d}} \right) \Vert u_{\alpha, n}-\1_D\Vert_{L^2}^{2/d}.\end{align*}
Combining this with the estimate of Lemma \ref{lem:L2rates}, which can be applied since a bounded set with Lipschitz boundary has finite perimeter and
\[\alpha \gs C\|n\|_{L^p} \gs C|\Omega|^{(p-2)/2p}\|n\|_{L^2},\]
we obtain
\begin{equation}\label{eq:markovhausest}\sup_{x \in \partial D} d\big(x, \partial \{u_{\alpha, n} > \theta \}\big) \ls C(\Omega, D) \max\left( \frac{1}{\theta^{2/d}},\ \frac{1}{(1-\theta)^{2/d}} \right) \alpha^{1/d}.\end{equation}
For the other half of the Hausdorff distance, we have by recalling \eqref{eq:Rbound1} and \eqref{eq:Rbound2} (whose validity determines $\alpha_0(D, \Omega, \theta)$ through \eqref{eq:spanishbound} and \eqref{eq:Xibound} which quantify the boundary effects), and applying Theorem \ref{thm:hausmaxprinc} that
\[\sup_{x \in \partial \{u_{\alpha, n} > \theta \}} d(x, \partial D) \ls C(\Omega, D) \max\left( \frac{1}{\theta^{1/2s}},\ \frac{1}{(1-\theta)^{1/2s}} \right)\alpha^{1/2s},\]
which, since $s \in (0, 1/2)$, $d\gs 2$ and we can assume $\alpha \in (0,1)$, can be combined with \eqref{eq:markovhausest} to get \eqref{eq:hausrate}.
\end{proof}

\begin{remark}\label{rem:noreallaplacianrates}
This method will not work when regularizing with the usual Laplacian, since we have used $s< 1/2$, which permits $\1_D \in H^s(\R^d)$ and the application of Lemma \ref{lem:L2rates}. It is easily adapted to the total variation, though: under the source condition and for adequate choices of exponents in the data term, one has density estimates for the level sets of solutions with uniform constant along the sequence and level $\theta$, and also holding for the limit $\1_D$ (see \cite[Thm.~4.5]{IglMer20} for a result in any dimension). In that case, the density estimates are also true beyond denoising, but to follow the argument of the proof of Theorem \ref{thm:convrates} we start with a rate of convergence in $L^2$ norm, which again is in principle only easily obtained in the denoising case (see \cite{Val21} for some slightly more general settings where it also holds, though).
\end{remark}

\section{Geometric convergence to continuous ideal data}\label{sec:convtocont}
We now consider Hausdorff convergence when the ideal data is continuous, so that it becomes flatter by the same rescaling used for Theorem \ref{thm:hausmaxprinc}, and with nontrivial operators $A$. This section is (except for Corollary \ref{cor:hausmaxprinc3} regarding denoising) devoted to the proof of the following result:
\begin{theorem}\label{thm:hausmaxprincA}
Let $\Omega$ be an open bounded set with Lipschitz boundary and satisfying the exterior ball condition, and the operator $A$ be such that $A^\ast$ maps $\Hm$ to $L^q(\Omega)$ with $q>d/2s$ compactly as well as $\Hm$ to $L^\infty(\Omega)$ continuously. Moreover, assume that for $u^\dag \gs 0$ the source condition $A^\ast z \in \partial [\frac12 |\cdot|^2_{H^s}](u^\dag)$ holds. 

Then, for $u_{\alpha, n}$ the unique minimizer of \eqref{eq:mingen} and for almost every $\theta >0$ and all $\eps >0$, we have  as $\|n\|_\Hm \to 0$ and $\alpha \to 0$ and with $\|n\|_\Hm/\alpha \ls C$ that
\[\sup_{x \in \{u_{\alpha, n} > \theta \}} d\big(x, \{u^\dag \gs \theta - \eps\}\big) \to 0, \text{ and }\, \sup_{x \in \{u_{\alpha, n} \ls \theta \}} d\big(x, \{u^\dag \ls \theta + \eps\}\big) \to 0.\]
Moreover, if for a.e.~$\theta \in \R \setminus \{0\}$ we have that $\{u^\dag > \theta\}$ satisfies inner and outer density estimates in the sense of \eqref{eq:densityestA}, then the Hausdorff convergence
\[d_H\big(\partial \{u_{\alpha, n} > \theta\}, \partial \{u^\dag > \theta\}\big) \to 0\]
holds for a.e.~$\theta \in \R \setminus \{0\}$ for which
\[\lim_{\eps\to 0} d_H\big(\{u^\dag \gs \theta - \eps\}, \{u^\dag > \theta\}\big) = 0,\]
that is, we need to exclude levels $\theta$ corresponding to `flat' parts of $u^\dag$.
\end{theorem}
\begin{proof}
It follows by combining Propositions \ref{prop:uniformconv} and \ref{prop:hausmaxprinc2} below. Note that because $A^\ast$ maps into $L^\infty(\Omega)$, and that the source condition is equivalent to the weak formulation \eqref{eq:weakforsubg} (with $v = A^\ast z$) of a Dirichlet problem with zero outer/boundary condition, we can apply Proposition \ref{prop:boundaryregularity} to infer that $u^\dag$ is H\"older continuous up to the boundary and $u^\dag=0$ on $\partial \Omega$.
\end{proof}

\begin{remark}\label{rem:sourcecondtv}
For total variation regularization and assuming the source condition the same convergence also holds, also for more general cases of operators: this is the main result of both \cite{IglMerSch18} and \cite{IglMer20}.
\end{remark}

We first show that under these assumptions, we can also treat regularization with nontrivial operators as if it were denoising. 
\subsection{From inversion of strongly smoothing operators to denoising}\label{sec:denoisingtrick}
As in Section \ref{sec:optimality}, the optimality condition reads
\begin{equation}
v_{\alpha, n} := -\frac{1}{\alpha} A^\ast \big(Au_{\alpha, n}-f-n\big) \in L^q(\Omega) \text{ with }q>\frac{d}{2s}, \text{ and }(-\Delta)^s u_{\alpha,n}=v_{\alpha, n}\text{ on }\Omega.
\end{equation}
Our goal is to transform this into a denoising problem for a perturbed measurement. To this end, we define
\begin{equation}\label{eq:regularizationbydenoising}
f_{\alpha, n} := \alpha v_{\alpha, n} + u_{\alpha, n}, \text{ so that } (-\Delta)^s u_{\alpha,n}=\frac{1}{\alpha}\big(f_{\alpha, n} - u_{\alpha, n}\big) \text{ in } \Omega,
\end{equation}
and notice that the last equation means that
\[u_{\alpha,n} = \argmin_{u\in L^2(\Omega)} \int_{\Omega} (u-f_{\alpha,n})^2 + \alpha \vert u \vert_{H^s}^2\]
where the denoising problem with $L^2$ fidelity term and data $f_{\alpha, n}$ can also be posed to lead to \eqref{eq:regularizationbydenoising}, since the Sobolev inequality \eqref{eq:sobineq} can be used and
\[p<2<\frac{2d}{d-2s}, \text{ so }f_{\alpha, n} \in L^2(\Omega).\]

\begin{prop}\label{prop:uniformconv}
Assume that $A^\ast$ maps $\Hm$ to $L^q$ compactly as well as $\Hm$ to $L^\infty$ continuously, the source condition $A^\ast z \in \partial [\frac12 |\cdot|^2_{H^s}](u^\dag)$ and the parameter choice $\|n\|_\Hm/\alpha \ls C$ holds. Then for the $f_{\alpha, n}$ defined in \eqref{eq:regularizationbydenoising} we have
\begin{equation}\label{eq:Linftyconv}\big\|f_{\alpha, n} - u^\dag\big\|_{L^\infty} \to 0 \text{ as }\alpha,\|n\|_\Hm \to 0.\end{equation}
\end{prop}
\begin{proof}
Given the source condition, the parameter choice, and the assumed compactness of $A^\ast$ into $L^q$, by Proposition \ref{prop:convergenceoflaplacians} we have $v_{\alpha, n} \to A^\ast z$ strongly in $L^q$. By the Stampacchia boundedness estimate of Proposition \ref{prop:boundedsol} this implies also $\|u_{\alpha, n} - u^\dag\|_{L^\infty} \to 0$. Moreover, in the proof of Proposition \ref{prop:convergenceoflaplacians} we also saw that the set $\{(Au_{\alpha, n}-f)/\alpha\}_\alpha$ is bounded in $\Hm$, so using the assumption that $A^\ast$ maps continuously into $L^\infty$ we have that $\|v_{\alpha, n}\|_{L^\infty} \ls C$. Combined with the previous convergence and in light of \eqref{eq:regularizationbydenoising}, we end up with \eqref{eq:Linftyconv}.
\end{proof}

\begin{example}\label{ex:convolutions}Let us assume that $p,q$ satisfy $p^{-1}+q^{-1}=1$ and $q>d/2s$, and that we have a convolution operator $A:L^p(\Omega) \to L^2(\Omega + \Sigma) = \Hm$ defined by
\[Au(x)=\int_{\R^d} K(x-y) u(y) \dd y \text{ with }K \in L^q(\Sigma),\]
where $\Sigma \subset \R^d$ is bounded. Then by the H\"older inequality we have \[\|Au\|_{L^\infty(\Omega + \Sigma)} \ls \|K\|_{L^q(\Sigma)}\|u\|_{L^p(\Omega)},\]
so $A$ maps into $L^\infty$ continuously. Now, the adjoint is the operator $A^\ast: L^2(\Omega + \Sigma) \to L^q(\Omega)$ given by
\begin{equation}\label{eq:adjointformula}A^\ast g(x) =\1_\Omega(x) \cdot \int_{\R^d} K(y-x) g(y) \dd y, \end{equation}
which analogously is bounded into $L^\infty(\Omega)$ on $L^2(\Omega + \Sigma) \subset L^p(\R^d)$. Moreover, $A^\ast$ is compact from $\Hm \to L^q(\Omega)$ as well. To see this, consider a $L^2$-weakly converging sequence $g_n \rightharpoonup g$ and notice that by \eqref{eq:adjointformula} and $K \in L^q(\Sigma) \subset L^2(\Sigma)$, we have that $A^\ast g_n(x) \to A^\ast g(x)$ for a.e.~$x \in \Omega + \Sigma$ as well. Moreover since $g_n$ is bounded in $L^2$ and $\Sigma$ is bounded, we also have 
\[|(A^\ast g_n)(x)|^2 \ls C \|K\|_{L^2(\Sigma)}^2 \ls C \|K\|_{L^q(\Sigma)}^2,\]
and since all the $A^\ast g_n$ have a common compact support we may use a constant function in the dominated convergence theorem. 
\end{example}

\begin{example}\label{ex:inversesourceprobs}
Another class of examples which fits these assumptions are inverse source problems in which the operator $A$ maps the right hand side of a linear elliptic PDE problem on a bounded domain $\Omega \subset \R^d$ with smooth boundary to its solution.

For instance, in the most basic case of an elliptic Dirichlet problem, the adjoint is again of the same type \cite[Lem.~2.24]{Tro10}. This means that for $d=2,3$ we also have that $A^\ast$ is continuous from $\mathcal{H}=L^2(\Omega)$ into $L^\infty(\Omega)$, by the classical Stampacchia boundedness estimate analogous to Proposition \ref{prop:boundedsol}, which holds for right hand side in $L^r(\Omega)$ with $r > d/2$. Compactness of $A^\ast$ into $L^q(\Omega)$ follows by elliptic regularity and Rellich-Kondrachov for $H^2(\Omega)$ automatically when $d \ls 4$, and otherwise under the additional condition
\[\frac{2d}{d-4}>q>\frac{d}{2s}, \text{ that is }s > \frac{d}{4}-1.\]
\end{example}

\subsection{Denoising with uniformly converging data}\label{sec:flat}
With Section \ref{sec:denoisingtrick} and Proposition \ref{prop:uniformconv} in mind, we now treat the case of denoising in which the data is uniformly continuous and uniformly converging. The structure of the proofs is analogous to those of Lemmas \ref{lem:compalphaone} and \ref{lem:compalphaoneD} and Theorem \ref{thm:hausmaxprinc}, but now we need to take care of the modulus of continuity of the data:
\begin{lemma}\label{lem:compalphaonef}
Let $\Om \subset \R^d$ be open but otherwise arbitrary, $\f \in C\big(\overline{\Om}\big)$ with $\f \gs 0$, $\max_{\Om} \f = 1$, and modulus of continuity $\omega_{\f}$. Assume that $\Vert \f_\alpha - \f \Vert_{L^\infty} \to 0$ as $\alpha \to 0$ and let $\u_\alpha$ be the unique minimizer of 
\begin{equation}
u \mapsto \int_{\Om} \big(u(x) - \f_\alpha(x)\big)^2 \dd x + |u|^2_{H^s(\Om)}.
\label{eq:onefracroff}
\end{equation}
Then for any $R > R_0:=R_0(d,s)$ as defined in Lemma \ref{lem:barrier}, a point $x_0 \in \Om$ with $d(x_0, \partial \Om)>R$, radius $0<r<R$, level $\otheta > \f(x_0)$ and tolerance $\eta < \otheta - \f(x_0)$ satisfying
\begin{equation}\label{eq:boundconditionf}C_B(R+1-r)^{-2s} \ls \otheta-\f(x_0)-\eta \text{ and }\omega_{\f}(R) \ls \frac{\eta}{2}\end{equation}
where $C_B:=C_B(d,s,1)$, we have that whenever 
\begin{equation}\label{eq:noisetolerancef}\|\f_\alpha - \f\|_{L^\infty} < \frac{\eta}{2},\end{equation}
also
\begin{equation}\label{eq:levelsetnotttheref}\u_\alpha \ls \otheta\text{ on }B(x_0,r), \text{ and }\  \u_\alpha(x_0) < \otheta.\end{equation}
\end{lemma}
In condition \eqref{eq:boundconditionf}, having a fixed lower bound on $R$ as well as an upper bound for it depending on the modulus of continuity might seem very restrictive, and it is. This means that we can only apply this lemma to functions that `change very slowly', which will be precisely the case when $\f$ is obtained from a uniformly continuous function by rescaling with a small factor, as we do in Proposition \ref{prop:hausmaxprinc2} below.
\begin{proof}[Proof of Lemma \ref{lem:compalphaonef}]
Let $w$ be again the translation to $x_0$ of the barrier of Lemma \ref{lem:barrier}, this time with $a = \f(x_0)$ and $b=\max_{\Om}\f=1$. Moreover, we first consider $\hat f$ instead of $\f_\alpha$ (the solution is denoted by $\u$). Then, on $B(x_0,R)$, we have
\[-(-\Delta)^s w \ls w - \f(x_0) \]
as well as the Euler Lagrange equation
\[ (-\Delta)^s \u = \f-\u \text{ on } B(x_0, R). \]
Summing up, we obtain
\[ (-\Delta)^s (\u-w) - \f +\u-w+\f(x_0) \ls 0,\]
which, using $\f \ls \f(x_0) + \omega_{\f}(R)$ on $B(x_0, R)$, writes
\[ (-\Delta)^s (\u-w) - \omega_{\f}(R) +\u- w \ls 0,\]
or
\[ (-\Delta)^s \left(\u-w - \omega_{\f}(R) \right) +\left( \u- w - \omega_{\f}(R)\right) \ls 0,\]
so that we can apply the maximum principle to $\u-w - \omega_{\f}(R)$ on $B(x_0, R)$ and take into account that by Lemma \ref{lem:threshold} we have $\u-w \ls 0$ on $\R^d \setminus B(x_0, R)$, to obtain that
\[\u(x) \ls w(x) + \omega_{\f}(R) \ls w(x) + \frac{\eta}{2} \text{ for all }x \in B(x_0, R).\]
To conclude, notice that as soon as $ \f-\epsilon\ls \f_\alpha \ls \f+\epsilon$, we have thanks to Corollary \ref{cor:comppple} and taking into account that the constant function with value $\epsilon$ solves $(-\Delta)^s u + u = \epsilon$ with itself as outer/boundary condition in any domain, also $\Vert \u - \u_\alpha\Vert_{L^\infty} \ls \epsilon$. This implies, thanks to \eqref{eq:boundconditionf}, the inequality
\[\Vert \u - \u_\alpha\Vert_{L^\infty} < \frac \eta 2\]
but then the last inequality combined with \eqref{eq:barriervalbound} implies, since $C_B$ is nondecreasing in the amplitude of the barrier, that
\begin{equation}\label{eq:wboundwithetaf}\begin{aligned}w &\ls C_B\big(d, s, 1-\f(x_0)\big)(R+1-r)^{-2s} \\&\ls C_B\big(d, s, 1\big)(R+1-r)^{-2s} < \otheta - \eta \ \;\text{ on }B(x_0, r) \subset B(x_0, R),\end{aligned}\end{equation}
so that
\[B(x_0, r) \cap \{\u_\alpha \gs \otheta\} = \emptyset,\text{ and }\u_\alpha(x_0)<\otheta.\qedhere\]
\end{proof}

We will also need the `flipped' comparison, analogous to Lemma \ref{lem:compalphaoneD}:
\begin{lemma}\label{lem:compalphaoneDf}
In the situation of Lemma \ref{lem:compalphaonef} but with $\f(x_0) >\otheta$, $\eta < \f(x_0) - \otheta$ and replacing \eqref{eq:boundconditionf} by
\begin{equation}\label{eq:flippedboundconditionf}C_B(R+1-r)^{-2s} \ls \f(x_0)-\otheta-\eta \text{ and }\omega_{\f}(R) \ls \frac{\eta}{2}\end{equation} we have that
\begin{equation}\label{eq:levelsetnottthereDf}\u_\alpha \gs \otheta - (1-\Xih) \text{ on }B(x_0,r)\text{ and }\u_\alpha(x_0) > \otheta - \big(1-\Xih(x_0)\big).\end{equation}
\end{lemma}
\begin{proof}
Similarly to Lemma \ref{lem:compalphaoneD}, we consider denoising of $1-\f_\alpha$ with solution $\Xih - \u_\alpha$. Since the function $1-\f_\alpha$ again takes values in $[0, 1]$, we can compare with the same barrier and follow the same argument to obtain a conclusion for the level $1-\otheta$ of $\Xi - \u_\alpha$. Notice that since we are comparing the possible value $1-\otheta$ with the input value $1-\f_\alpha(x_0)$, it is $1-\otheta - (1-\f(x_0))=\f(x_0)-\otheta$ that appears in condition \eqref{eq:flippedboundconditionf}.
\end{proof}

Again undoing the rescaling, we arrive at:

\begin{prop}\label{prop:hausmaxprinc2}Let $\Omega$ be an open bounded set with Lipschitz boundary and satisfying the exterior ball condition and let $f \in \C(\overline{\Omega})$ with $f=0$ on $\partial \Omega$, $f_\alpha$ with
\begin{equation}\label{eq:linftyconv}\|f_\alpha - f\|_{L^\infty} \xrightarrow[\alpha \to 0]{} 0,\end{equation}
and $u_\alpha$ be the unique minimizer of 
\begin{equation}\label{eq:fracroff}
u \mapsto \int_{\Omega} \big(u(x) - f_\alpha(x)\big)^2 \dd x + \alpha |u|^2_{H^s}.
\end{equation}
then given $\theta \in \R \setminus \{0\}$ and any $\eps >0$ we have that
\begin{equation}\label{eq:halfhausdorff}\lim_{\alpha \to 0} \sup_{x \in \{u_\alpha > \theta \}} d\big(x, \{f \gs \theta - \eps\}\big) = 0, \text{ and }\, \lim_{\alpha \to 0} \sup_{x \in \{u_\alpha \ls \theta \}} d\big(x, \{f \ls \theta + \eps\}\big) = 0.
\end{equation}
%&\lim_{\alpha \to 0} \sup_{x \in \{u_\alpha \ls \theta \}} d\big(x, \{f \ls \theta - \eps\}\big) = 0, \quad \lim_{\alpha \to 0} \sup_{x \in \{u_\alpha > \theta \}} d\big(x, \{f \gs \theta + \eps\}\big) = 0 \quad\text{if }\theta < 0.
Moreover, if for a.e.~$\theta \in \R \setminus \{0\}$ we have that $\{f > \theta\}$ satisfies inner and outer density estimates in the sense of \eqref{eq:densityestA}, then the Hausdorff convergence
\begin{equation}\label{eq:hausconvwithf1}d_H\big(\partial \{u_\alpha > \theta\}, \partial \{f > \theta\}\big) \xrightarrow[\alpha \to 0]{} 0\end{equation}
holds for a.e.~$\theta \in \R \setminus \{0\}$ for which
%\begin{equation}\label{eq:notflat}\lim_{\eps\to 0} d_H\big(\{f \gs \theta - \eps\}, \{f > \theta\}\big) = 0, \text{ and } \lim_{\eps\to 0} d_H\big(\{f \ls \theta + \eps\}, \{f \ls \theta\}\big) = 0,\end{equation}
\begin{equation}\label{eq:notflat}\lim_{\eps\to 0} d_H\big(\{f \gs \theta - \eps\}, \{f > \theta\}\big) = 0.\end{equation}
\end{prop}
\begin{proof}
By linearity and taking into account Lemma \ref{lem:threshold}, we have that considering the positive part $f_\alpha^+$ as input leads to $u^+_{\alpha, 0}$, and $-f_\alpha^-$ to $-u^-_\alpha$. Therefore, without loss of generality we can assume $\theta >0$, $f\gs 0$, and prove just the first limit of \eqref{eq:halfhausdorff}. Moreover, since the result is not quantitative and the problem is linear, it is invariant by multiplication with a constant, and we can also assume $\max_\Omega f = 1$. With this in mind, our plan is to apply Lemmas \ref{lem:compalphaonef} and \ref{lem:compalphaoneDf}, which assume both upper and lower bounds on $R$, depending on the modulus of continuity of the function $\f$ that it is applied to. Here, we use $\f_\alpha(x) = f_\alpha(\alpha^{1/2s}x)$ and $\f(x) = f(\alpha^{1/2s}x)$, that is, with the same rescaling as applied to minimizers in Remark \ref{rem:scalingone} after Lemma \ref{lem:scaling}, and also in Theorem \ref{thm:hausmaxprinc}.

First, since $f \in\C(\overline{\Omega})$ it is uniformly continuous it has a modulus of continuity $\omega_f$, and by definition we have that
\[\omega_{\f}(R) = \omega_f(R\alpha^{1/2s}),\]
and for any fixed $R>0$ this quantity can be made as small as needed by reducing $\alpha$ further, ensuring in particular that for any given $\eta>0$, we can satisfy $\omega_{\f}(R) < \eta/2$. Moreover, we can also enforce $\|f_\alpha - f\|_{L^\infty} < \eta/2$, and this quantity is invariant to rescaling.

Now, if $x \in \{ f < \theta - \eps\}$ then we can apply Lemma \ref{lem:compalphaonef} with $\eta = (\theta - f(x))/2$ and $\otheta = \theta$ which gives us that for every $\eps> 0$
\begin{equation*}\begin{gathered}u_\alpha \ls \theta \text{ for all }x \in \{ f < \theta - \eps\} \text{ with }\max\left(d\big(x, \{ f \gs \theta - \eps \}\big),\,d(x, \partial \Omega)\right)>R\alpha^{1/2s},\\
\text{ where }R > R_\eps := \max\big( C \eps^{-1/2s} - C, R_0) > C(\theta-f(x))^{-1/2s}-C,\end{gathered}\end{equation*}
where $C$ is a constant that could be computed explicitly from \eqref{eq:boundconditionf}, directly implying that
\[\sup_{x \in \{u_\alpha > \theta \} \cap \{d(x, \partial \Omega)>R\alpha^{1/2s}\}} d\big(x, \{ f \gs \theta - \eps \}\big) \ls R \alpha^{1/2s}.\]
Here, we run into the same problem as in the proof of Theorem \ref{thm:hausmaxprinc}, because we need to guarantee that if $u_\alpha(x) \gs \theta$ then $d(x, \partial \Omega)>R\alpha^{1/2s}$. However, since we have assumed $f \in [0,1]$ we can use the estimate \eqref{eq:spanishbound} without modifications. Therefore, considering only $\alpha < \alpha_\theta$ we can remove the restriction and obtain the first limit in \eqref{eq:halfhausdorff}.

For the opposite inequality $u_\alpha > \theta$ which should hold for $x \in \{f > \theta + \eps\}$ away from the boundary of this set, as in Theorem \ref{thm:hausmaxprinc} we use the modified comparison Lemma \ref{lem:compalphaoneDf}, and Lemma \ref{lem:ctdecay} to quantify the boundary effect on the function $\Xi_\alpha$. First, notice that the condition $f = 0$ on $\partial \Omega$ implies that given $\theta$, there is $\delta_\theta$ such that $\omega_f(\delta_\theta) < \theta$, which implies
\begin{equation}\label{eq:fband}\{f>\theta\} \subset \{d(\cdot,\partial \Omega)>\delta_\theta\} =: \Omega_\theta.\end{equation}
In particular, we have $x \in \Omega_\theta$. Analogously to \eqref{eq:Xibound} we have now for all $y \in \Omega_\theta$ that
\begin{equation}\label{eq:Xiboundf}1-\Xi_\alpha(y)\ls C_B\big(1+\delta_\theta\,\alpha^{-1/2s}\big)^{-2s} := B_\theta(\alpha) \xrightarrow[\alpha \to 0]{} 0.\end{equation}
Note that $B_\theta(\cdot)$ depends only of $\theta$ and not on $x \in \{f>\theta + \eps\}$ and we may restrict $\alpha$ so that $B_\theta(\alpha) < (f(x)-\theta)/3$. To apply Lemma \ref{lem:compalphaoneDf}, we use the parameters $\otheta = \theta + (f(x)-\theta)/3$ and $\eta = (f(x)-\theta)/3$, additionally restrict $R$ so that
\[R>R_\eps':= \max( C'\eps^{-1/2s} - C', R_0 ) >C'(f(x)-\theta)^{-1/2s}-C',\]
where $C'$ is derived from \eqref{eq:flippedboundconditionf} and in general $C'\neq C$, to finally obtain
\[u_\alpha > \theta \, \text{ for all }x \in \{ f > \theta + \eps \} \text{ with }d\big(x, \{ f \ls \theta + \eps \}\big) > R\alpha^{1/2s},\]
and hence 
\[\sup_{x \in \{u_\alpha \ls \theta \}} d\big(x, \{ f \ls \theta + \eps \}\big) \ls R \alpha^{1/2s},\]
leading to the second limit in \eqref{eq:halfhausdorff}. 

%We reiterate that in the above estimates we have assumed that $\alpha$ is small enough so that $B_\theta(\alpha) < (f(x)-\theta)/3$, which is enough since if $x \notin \Omega_\theta$ then certainly $d\big(x,\{f \ls \theta + \eps \}\big)=0$.

Moreover, we have
\[d\big(x, \{ f > \theta \}\big) \ls d\big(x, \{ f \gs \theta - \eps \}\big) + d_H\big(\{ f \gs \theta - \eps \}, \{ f > \theta \}\big)\]
and analogously for $d\big(x, \{ f \ls \theta \}\big)$. If \eqref{eq:notflat} holds, then by diagonalization in $\alpha$ and $\epsilon$ we obtain
\begin{equation}\label{eq:quarterhausdorff1}\lim_{\alpha \to 0} \sup_{x \in \{u_\alpha > \theta \}} d\big(x, \{f > \theta \}\big) = 0.\end{equation}
On the other hand, because the $\{f \ls \theta + \eps\}$ are nested decreasing and closed (note that $f$ is continuous) and $\{f \ls \theta\}$ is their intersection, the limit $d_H\big(\{f \ls \theta + \eps\}, \{f \ls \theta\}\big) \xrightarrow[\eps \to 0]{} 0$ holds unconditionally as we show in Lemma \ref{lem:nestedhausdorff} below, so we also have
\begin{equation}\label{eq:quarterhausdorff2}\lim_{\alpha \to 0} \sup_{x \in \{u_\alpha \ls \theta \}} d\big(x, \{f \ls \theta \}\big) = 0.\end{equation}
Combining the limits \eqref{eq:quarterhausdorff1} and \eqref{eq:quarterhausdorff2} as in the proof of Theorem \ref{thm:hausmaxprinc} we have 
\[\sup_{x \in \partial \{u_\alpha > \theta \}} d\big(x, \partial \{f \gs \theta \}\big) \xrightarrow[\alpha \to 0]{} 0.\]

To complete the Hausdorff convergence of level sets $d_H\big(\partial \{u_\alpha > \theta \}, \partial \{ f > \theta \}\big) \to 0$ we also need the convergences
\[\sup_{x \in \partial \{ f > \theta \}} d\big(x, \partial \{u_\alpha > \theta \}\big) \xrightarrow[\alpha \to 0]{} 0,\]
and these are, as in Theorem \ref{thm:hausmaxprinc}, implied by combining the $L^1$ convergence of the level sets (that follows for a.e.~$\theta$ from $\|u_\alpha-f\|_{L^2} \to 0$) with the density estimates we have assumed for the level sets of $f$.
\end{proof}

\begin{remark}\label{eq:sard}
Condition \eqref{eq:notflat} is in some sense natural when trying to deduce convergence of level sets, since subsets where $f$ is constant are inherently unstable when applying any kind of smoothing to it. In particular, it also appears in \cite[Thm.~1.3]{IglMer21} for a similar result on $\TV$ regularization. A natural question arising from this condition is if it could fail for many levels of a given function simultaneously. In \cite[Ex.~5.10]{IglMer21} a function is constructed for which it fails for almost every $\theta$. This function belongs only to $\BV$ though, and here we deal with continuous data. On the other end of the spectrum, if $f \in \C^d$ (i.e. $d$ times continuously differentiable) the classical Morse-Sard theorem tells us that the set of critical values $\theta$ is of measure $0$ in $\R$, so \eqref{eq:notflat} holds for almost every $\theta$. In between these two extreme cases the situation is more intricate. In the recent work \cite[Thm.~1.2]{FerKorRov20} a Morse-Sard type result is proved in Bessel potential spaces, which on $\R^d$ (we may extend $f$ by zero since it's compactly supported) equal the Gagliardo-Slobodeckij spaces we consider here. Since $q>d/2s$ and assuming $s \in (1/2,1)$ this result can be applied to $u^\dag \in W^{2s, q}(\R^d)$, which is the best we could hope to obtain from the source condition (in fact such regularity is true only locally in $\Omega$, see \cite[Thm.~1.4 and Sec.~5]{BicWarZua17}), to conclude that for a.e.~$\theta$ the Hausdorff dimension of the set of critical points of $f$ with value $\theta$ is bounded above by $d-2s$. However, \eqref{eq:notflat} holds only when this set is empty, so it remains a genuine additional assumption that we cannot deduce from the regularization framework.
\end{remark}

% \begin{proof}[Proof of Theorem \ref{thm:hausmaxprincA}]
% The strategy is to combine Proposition \ref{prop:uniformconv} with a slight refinement of Proposition \ref{prop:hausmaxprinc2}. The difference is that in Proposition \ref{prop:hausmaxprinc2} we assumed that $f$ is compactly supported, for simplicity. However, for $u^\dag$ satisfying the source condition this will typically not be the case.

% Note that because $A^\ast$ maps into $L^\infty(\Omega)$, we can apply Proposition \ref{prop:boundaryregularity} to infer that $u^\dag$ is H\"older continuous.
% \end{proof}

Compared to the parameter choice in $L^q$ norm \eqref{eq:paramchoice} used in Theorem \ref{thm:hausmaxprinc}, the convergence \eqref{eq:linftyconv} assumed in the above result is in a different norm but without rate, so neither implies the other. Moreover, it is also not possible to fit plain denoising into Proposition \ref{prop:uniformconv}, since considering the identity map between $L^p(\Omega)$ and $L^2(\Omega)$ does not satisfy any of the assumptions on $A$. In any case, the result still holds with the noise and parameter conditions of Theorem \ref{thm:hausmaxprinc} and continuous limit:

\begin{cor}\label{cor:hausmaxprinc3}Let $\Omega$ be as in Proposition \ref{prop:hausmaxprinc2}, $f \in \C(\overline{\Omega})$ vanish on $\partial \Omega$, and $u_{\alpha, n}$ be the unique minimizer of 
\begin{equation}\label{eq:fracroff3}
u \mapsto \int_{\Omega} \big(u(x) - f(x)-n(x)\big)^2 \dd x + \alpha |u|^2_{H^s}.
\end{equation}
Then, if the parameter choice is such that
\begin{equation}\label{eq:paramchoiceagain}\frac{\|n\|_{L^q}}{\alpha} \to 0 \text{ for some } q > \frac{d}{2s},\end{equation}
then we have the same conclusions as in Proposition \ref{prop:hausmaxprinc2}.
\end{cor}
\begin{proof}
As in Theorem \ref{thm:hausmaxprinc}, we can control the effect of the noise in $L^\infty$ norm, that is, we can obtain
\[\|u_{\alpha, n} - u_{\alpha, 0}\|_{L^\infty(\Omega)} \ls C_S(\Omega, q, d, s) \left(\frac{\|n\|_{L^q(\Omega)}}{\alpha}\right)^{2d/(d-2s)} \text{ for each } q > \frac{d}{2s}.\]
so by reducing $\alpha$, for any given $\eps >0$ we can enforce
\begin{equation}\|u_{\alpha, n} - u_{\alpha, 0}\|_{L^\infty(\Omega)} \ls \frac{\eps}{2}.\label{eq:smallnoise}\end{equation}
For the rest of the proof, we apply first Proposition \ref{prop:hausmaxprinc2} to the result of the denoising problem
\[u \mapsto \int_{\Omega} \big(u(x) - f(x)\big)^2 \dd x + \alpha |u|^2_{H^s}.\]
(so with $f_\alpha = f$ for all $\alpha$) which directly implies
\begin{equation}\lim_{\alpha \to 0} \sup_{x \in \{u_{\alpha, 0} > \theta \}} d\left(x, \left\{f \gs \theta - \frac{\eps}{2}\right\}\right) = 0, \text{ and } \lim_{\alpha \to 0} \sup_{x \in \{u_{\alpha, 0} \ls \theta \}} d\left(x, \left\{f \ls \theta + \frac{\eps}{2}\right\}\right) = 0.
\end{equation}
Now, we only need to note that \eqref{eq:smallnoise} also implies
\[ \{u_{\alpha, 0} > \theta \} \supset \left\{u_{\alpha, n} \gs \theta + \frac{\eps}{2} \right\} \quad \text{and} \quad  \{u_{\alpha, 0} \ls \theta \} \supset \left\{u_{\alpha, n} \ls \theta - \frac{\eps}{2} \right\}.\]
This last inclusion of sets implies then, for any fixed $\theta$ and $\eps$,
\begin{align*}&\lim_{\alpha \to 0} \sup_{x \in \{u_{\alpha, n} > \theta + \frac{\eps}{2} \}} d\left(x, \left\{f \gs \theta - \frac{\eps}{2}\right\}\right) = 0 \text{ and }\\ &\lim_{\alpha \to 0} \sup_{x \in \{u_{\alpha, n} \ls \theta - \frac{\eps}{2} \}} d\left(x, \left\{f \ls \theta + \frac{\eps}{2}\right\}\right) = 0,
\end{align*}
which changing $\theta$ into $\theta \pm \eps/2$ provides
\[\lim_{\alpha \to 0} \sup_{x \in \{u_{\alpha,n} > \theta \}} d\big(x, \{f \gs \theta - \eps\}\big) = 0, \text{ and } \lim_{\alpha \to 0} \sup_{x \in \{u_{\alpha,n} \ls \theta \}} d\big(x, \{f \ls \theta + \eps\}\big) = 0 \]
which is \eqref{eq:halfhausdorff} and allows to conclude similarly as in Proposition \ref{prop:hausmaxprinc2}.
\end{proof}

Let us now state the lemma for nested sets used in the proof of Proposition \ref{prop:hausmaxprinc2}.

\begin{lemma}\label{lem:nestedhausdorff}
Let $E_k \subset \R^d$ be nonempty, closed, bounded, and nested decreasing, that is $E_{k+1}\subset E_k$ for all $k\gs 1$. Then denoting 
\[E := \bigcap_{k=1}^\infty E_k, \text{ we have }d_H(E_k, E)\xrightarrow[k \to \infty]{} 0.\]
\end{lemma}
\begin{proof}
We start with the definition
\[d_H(E_k, E) = \max \left( \sup_{x \in E_k} d(x,E)\,,\,\sup_{x \in E} d(x,E_k)\right).\]
Since $E$ is the intersection of all $E_k$, the second argument in the maximum is zero. Assume now for a contradiction that there is $\delta > 0$
\[x_k \in E_k \text{ with }d(x_k, E) > \delta \text{ for all }k.\]
Now, since the $E_k$ are nested we have that for any fixed $k_0$, then $x_{k} \in E_{k_0}$ for all $k \gs k_0$. Since $E_{k_0}$ is bounded and closed, up to a subsequence $x_k$ converges to some point in $E_{k_0}$. But since this is the case for any $k_0$, we could by diagonalization (that is, recursively) find a subsequence converging to some $x_0$ with $x_0 \in E_k$ for all $k\gs 1$. Then $x_0 \in E$, but also $d(x_0,E)\gs \delta$, which is impossible.
\end{proof}

% Combining the techniques of this section and the ones used to arrive at Corollary \ref{cor:pwconst}, one directly obtains an analogous result for piecewise continuous functions:

% \begin{cor}\label{cor:pwcontinuous}
% Assume that $f \gs 0$ is piecewise continuous and compactly supported on $\Omega$ and attains the value $0$, that is, there are $f_0 \in \C(\Omega)$ with compact support, $f_0 \gs 0$ and $\min_\Omega f_0 = 0$, $0= c_0 < \ldots < c_N$ and $\Omega \supset \Omega_1 \supset \ldots \supset \Omega_{N}$ such that
% \begin{equation}\label{eq:pwccf} f=f_0 + \sum_{i=1}^{N} \big( c_{i}-c_{i-1} \big) \Omega_i,\end{equation}
% where the boundaries $\partial \Omega_i$ are Lipschitz. Assume moreover that and $n, \alpha$ are such that the parameter choice condition \eqref{eq:paramchoiceagain} holds. Then for $u_{\alpha, n}$ the minimizers of \eqref{eq:fracroff3} and a.e.~$\theta \in (c_{i-1}, c_{i})$ with $i=1, \ldots, N$ we have the convergence
% \[d_H\big(\partial \{u_{\alpha, n} > \theta\}, \partial\{f > \theta\}\big) \xrightarrow[\alpha \to 0]{} 0.\]
% \end{cor}

\section*{Acknowledgments} A large portion of this work was completed while the first-named author was employed at the Johann Radon Institute for Computational and Applied Mathematics (RICAM) of the Austrian Academy of Sciences (\"OAW), during which his work was partially supported by the State of Upper Austria.

\bibliographystyle{plain}
\bibliography{conv_fractional_laplacian}
\end{document}